\documentclass[11pt,a4paper]{amsart}
\AddToHook{cmd/section/before}{\clearpage}
\usepackage[left=2.5cm,right=2.5cm,top=2.5cm,bottom=2.5cm]{geometry}
\usepackage[utf8]{inputenc}
\usepackage{bm,amsmath, amsthm,graphicx,amsfonts,tikz,enumerate,slashed}
\usepackage{appendix,verbatim}
\usepackage{hyperref}

\title{$G_2$-instantons on the ALC members of the $\Bsev$ family}
\usepackage[foot]{amsaddr}
\author{Jakob Stein}
\address{University of Campinas, University of Bath}
\email{jstein@unicamp.br}
\author{Matt Turner}
\date{}

\theoremstyle{definition}

\theoremstyle{plain}
\newtheorem{theorem}{Theorem}[section]
\newtheorem{corollary}[theorem]{Corollary}
\newtheorem{lemma}[theorem]{Lemma}

\newtheorem{prop}[theorem]{Proposition}
\newtheorem{proposition}[theorem]{Proposition}
\newtheorem{thmx}{Theorem}

\newtheorem{remark}[theorem]{Remark}
\theoremstyle{plain}

\DeclareMathOperator{\SU}{SU}
\DeclareMathOperator{\su}{\mathfrak{su}}
\DeclareMathOperator{\Imom}{Im\,\Omega}

\begin{document}

\newcommand{\G}{\SU(2)^2\times\textup{U}(1)}
\newcommand{\Bsev}{\mathbb{B}_7}
\newcommand{\Csev}{\mathbb{C}_7}
\newcommand{\Dsev}{\mathbb{D}_7}
\newcommand{\imH}{\text{im}(\mathbb{H})}
\newcommand{\Ad}{\text{Ad}}
\newcommand{\Hom}{\text{Hom}}
\newcommand{\R}{\mathbb{R}}
\newcommand{\Z}{\mathbb{Z}}
\newcommand{\Dirac}{\slashed{\text{D}}}
\newcommand{\Reom}{\text{Re}\,\Omega}
\newcommand{\Pid}{P_{\textup{id}}}

\newcommand{\f}{\bm{f}}
\newcommand{\g}{\bm{g}}
\newcommand{\h}{\bm{h}}

\newcommand*{\vertchar}[2][0pt]{%
  \tikz[
    inner sep=0pt,
    shorten >=-.15ex,
    shorten <=-.15ex,
    line cap=round,
    baseline=(c.base),
  ]\draw
    (0,0) node (c) {#2}
    ($(c.south)+(#1,0)$) -- ($(c.north)+(#1,0)$);%
}
\setlength{\parskip}{.4cm}
\setlength{\parindent}{0cm}

\begin{abstract}
    Using co-homogeneity one symmetries, we construct a two-parameter family of non-abelian $G_2$-instantons on every member of the asymptotically locally conical $\mathbb{B}_7$-family of $G_2$-metrics on $S^3 \times \mathbb{R}^4 $, and classify the resulting solutions. These solutions can be described as perturbations of a one-parameter family of abelian instantons, arising from the Killing vector-field generating the asymptotic circle fibre. Generically, these perturbations decay exponentially to the model, but we find a one-parameter family of instantons with polynomial decay. Moreover, we relate the two-parameter family to a lift of an explicit two-parameter family of anti-self-dual instantons on Taub-NUT $\mathbb{R}^4$, fibred over $S^3$ in an adiabatic limit. 
\end{abstract}
\maketitle
\section{Introduction}
Motivated by the conjectural picture outlined in \cite{donaldson1996gauge, donaldson2009gauge}, a number of  constructions of special Yang-Mills instantons on compact $G_2$-manifolds have appeared in the literature e.g. \cite{walpuski:g2instantons, walpuski:tcsinstantons}. In particular, if $(M, \varphi)$ is a $G_2$-manifold, equipped with a principal $G$-bundle $P \rightarrow M$ for a compact Lie group $G$, a connection $A$ on $P$ is called a $G_2$-\textit{instanton} if it satisfies the $G_2$-\textit{instanton equations}, namely 
\begin{align} \label{eq:g2instantons00}
F_A \wedge * \varphi = 0 
\end{align}
where $*$ is the Hodge star of the Riemannian metric defined by $\varphi \in \Omega^3 (M)$, and $F_A \in \Omega^2 \left(\mathrm{ad} P \right)$ is the curvature of $A$. These $G_2$-instantons generalise the \textit{anti-self-dual} (ASD) Yang-Mills instantons found in dimension four, which can  also appear as solutions of \eqref{eq:g2instantons00} in re-scaled limits along associative submanifolds of $(M, \varphi)$, see \cite{tian2000gauge}.

In this paper, we consider $G_2$-instantons on a family of non-compact $G_2$-manifolds, asymptotic to a circle fibration over a 6-dimensional Calabi-Yau cone. Geometries of this type are referred to as \textit{asymptotically locally conical} (ALC) in \cite{FoscoloALC}, and generalise the asymptotically locally flat (ALF) hyperk\"{a}hler metrics appearing in dimension four. Thanks to the  construction in \cite{FoscoloALC}, complete families of ALC $G_2$-metrics exist in abundance: given an asymptotically conical (AC) Calabi-Yau 3-fold and a suitable non-trivial circle bundle, \cite{FoscoloALC} produces a one-parameter family of circle-invariant ALC $G_2$-metrics on the total space. The family is  parameterised by the asymptotic length $\ell$ of the circle fibre, and collapses with bounded curvature back to the AC Calabi-Yau in the limit $\ell \rightarrow 0$. The circle-invariant $G_2$-structures produced by \cite{FoscoloALC} admit natural abelian solutions to \eqref{eq:g2instantons00},  by considering the harmonic one-form given by the metric dual of the Killing vector-field generating the circle action, c.f. \cite[Proposition 2.15]{lotay:g2conifolds}.   

On the other hand, circle-invariant solutions of \eqref{eq:g2instantons00} on these manifolds are closely related to Hermitian Yang-Mills connections on the AC Calabi-Yau in the collapsed limit, but we will not pursue this line of investigation here; further details will appear elsewhere. Instead, we will restrict our focus to the only known family of ALC $G_2$-metrics that do not collapse with bounded curvature as $\ell \rightarrow 0$, referred to as the $\Bsev$-family in \cite{gibbonspope:g2mtheory}, on $S^3 \times \R^4$. Its existence was predicted in \cite{bggg}, and it was first constructed in \cite{bogoyavg2}, c.f. \cite{FHN18}. The associated one-parameter family of $G_2$-structures also possess \textit{co-homogeneity one} symmetries, i.e. there is a compact Lie group of symmetries, in this case $SU(2)^2 \times U(1)$, with generic orbits of co-dimension one. Further discussions on the geometry of this family, and historical remarks, can be found in \cite{ST23}.  

Exploiting these co-homogeneity one symmetries, a two-parameter family of $SU(2)^2 \times U(1)$-invariant $G_2$-instantons on the $\Bsev$-family, with gauge group $SU(2)$, was found by Lotay-Oliveira in \cite{LO18b}. These appear as (non-explicit) solutions to the ODE system resulting from \eqref{eq:g2instantons00} with the enhanced symmetry. Their examples can be viewed in the following way: as before, consider the vector-field $X$ generated by the $U(1)$-symmetry, which has period $2\pi$ at infinity. Using the freedom to scale by a constant $\kappa$, one obtains a one-parameter family $\kappa d X$ of two-forms solving \eqref{eq:g2instantons00}. For each fixed $\kappa>1$, these abelian solutions can be perturbed to a one-parameter family of non-abelian instantons asymptotic to $\kappa X$ at some exponential rate.  

However, the analysis in \cite{LO18b} is valid only sufficiently close to the only ALC metric of the $\Bsev$ family that is known explicitly, namely the metric constructed by Brandhuber–Gomis–Gubser–Gukov (BGGG) in \cite{bggg}, and fails to determine which of the abelian instantons are rigid as invariant solutions of \eqref{eq:g2instantons00} for the generic member of the family, see Remark \ref{remark:LO}. Moreover,  even for the BBBG metric, their analysis fails to capture the full space of solutions. 

\subsection*{Main results and plan of the paper} In this paper, we revisit the study of $SU(2)^2\times U(1)$-invariant $G_2$-instantons on the $\Bsev$-family appearing in \cite{LO18b} using more sophisticated methods, and classifying solutions to the resulting ODE system, for every member of the $\Bsev$-family. In doing so, we construct new non-abelian instantons on the $\Bsev$-family that have curvature decay at polynomial rate normal to the asymptotic circle fibre, and have trivial holonomy around this circle at infinity. These appear at the boundary of the space of solutions with exponential decay, generalising results appearing in the thesis of the second author for the BGGG metric.  

To facilitate our study of the ODE system, which can be found as \eqref{eq:B7odes} in the text, we begin with some preliminaries on the $\Bsev$-family in \S \ref{section:G2define}. As well as proving some useful inequalities, we write down the metric coefficients appearing in \cite{LO18b} in terms of the coefficients of the $G_2$-structure in \cite{FHN18}. 

We study $SU(2)^2 \times U(1)$-invariant $G_2$-instantons on these manifolds in \S \ref{section:G2overview}. The $SU(2)^2 \times U(1)$-equivariant $SU(2)$-bundles over the $\Bsev$-family admitting irreducible invariant connections are classified: coinciding with the classification up to $SU(2)^2$-equivariance in \cite{LO18b}. There only two such equivariant bundles on $\Bsev$-family, referred to $P_1$, $P_\mathrm{id}$, corresponding to the trivial and non-trivial lifts of the singular isotropy subgroup to the total space. As in \cite{LO18b}, we focus on studying invariant instantons on $P_1$. On this bundle, the ODE system \eqref{eq:B7odes} simplifies, and the additional $U(1)$-action acts freely on the total space, so the associated one-parameter family of abelian instantons can be extended over the whole manifold. However, we note that it may be possible, using explicit solutions in certain adiabatic limits, to recover some solutions on $P_\mathrm{id}$, c.f. \S \ref{sec:taubnut}. 

We prove our main theorem in the remainder of \S \ref{section:G2overview}. We first recall from \cite{LO18b} that $SU(2)^2 \times U(1)$-invariant $G_2$-instantons on $P_1$ are in a two-parameter family $(f^+_1, g^+_1)$ in a neighbourhood of the singular orbit $S^3$, appearing as initial conditions for the ODE system. With this in mind, the main theorem can be stated as follows:  
\begin{thmx} \label{thm:A} Let $S^3 \times \R^4$ be equipped with the $\Bsev$-metric of asymptotic circle length $\ell>0$. Then $SU(2)^2 \times U(1)$-invariant $G_2$-instantons on $P_1$ are either in a one-parameter abelian family, or in an irreducible two-parameter family. Respectively, these correspond to initial conditions $(f^+_1, g^+_1)$ with $g^+_1 \geq \tfrac{1}{2}\ell^{-2}$,  $0<|f^+_1|\leq f(g^+_1)$ for some increasing function $f$ vanishing at $\tfrac{1}{2}\ell^{-2}$. 
\end{thmx}

\begin{figure}[h]
    \begin{center}
\begin{tikzpicture}[scale=1]
    % Define the function
    \draw[help lines, color=gray!30, dashed] (-0.5,-2.5) grid (5,2.5);

    \draw[domain=-2:2, smooth, variable=\y, ultra thick] plot ({\y*\y + 1}, \y) node[right]{$f(g^+_1)$};

    % Shade the area below the function
    \fill[red, opacity=0.2] (-0.5,-2) -- plot[domain=-2:2, smooth, variable=\y] ({\y*\y + 1}, \y) -- (-0.5,2) -- cycle;
   \fill[green, opacity=0.2] (5, -2) -- (5, 2) -- plot[domain=2:-2, smooth, variable=\y] ({\y*\y + 1}, \y) -- cycle;

    % Draw the axes
    \draw[->] (-0.5,0) -- (5.5,0) node[right] {$g_1^+$};
    \draw[->] (0,-2.5) -- (0,2.5) node[above] {$f_1^+$};
    \draw[blue, opacity=0.6, ultra thick] (-0.5,0) -- (5,0);
    \draw[dashed, opacity=0.8] (1,2.5) -- (1,-2.5);
    % Add labels
    \filldraw[blue, opacity=0.6] (1, 0) circle (2pt) ;
    \node[right] at (1, -2.5) {$g_1^+=\tfrac{1}{2} \ell^{-2}$}; 
    \node[below] at (1, 1.5){Incomplete}; 
    \node[below] at (3, 1) {Complete}; 
    \node[below] at (4, 0) {Abelian}; 
\end{tikzpicture}
\end{center}
    \caption{The region of initial conditions that lead to complete bounded solutions of \eqref{eq:g2instantons00} on the $\Bsev$-family. The $f_1^+ = 0$ axis corresponds to a one-parameter family of abelian solutions.}
    \label{fig:boundary}
\end{figure}
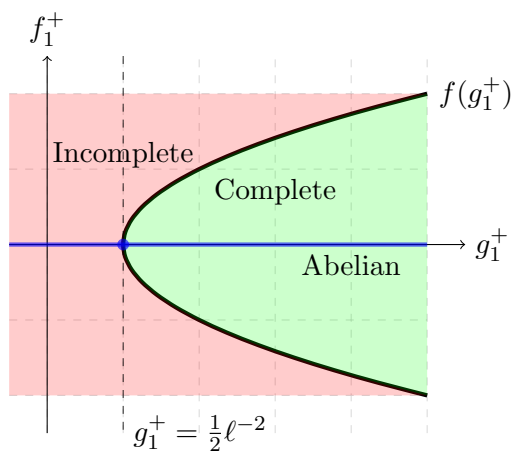

See also Theorem \ref{thm:4.3} in \S \ref{sec:propofsols}. The interior of this set corresponds to $G_2$-instantons with exponential decay to an abelian solution, while the boundary corresponds to a one-parameter family of $G_2$-instantons asymptotic to the abelian solution at $(0,\tfrac{1}{2}\ell^{-2})$ with a polynomial rate. This special abelian solution corresponds to a lift of the Killing vector-field generating the asymptotic circle.

To prove this theorem, we begin by extending some of the basic analytic results of \cite{LO18b} in \S \ref{sec:existresBGGG}, particularly the rigidity of the abelian solutions with $g^+_1< \tfrac{1}{2}\ell^{-2}$. We then describe possible asymptotics of complete solutions in \S \ref{sec:propofsols}. We find a continuous two-parameter family of end-solutions $(G_\infty, \lambda)$, $\lambda \in \R$, $G_\infty\geq \ell^{-1}$, as well as the line $(G_\infty, 0)$, $G_\infty \in \R$ corresponding to the one-parameter family of abelian solutions.

With this parametrisation, in \S \ref{sec:peturbation} we use a ``shooting from infinity'' argument inspired by \cite{FHN18}, to show that if an end solution with $G_\infty > \ell^{-1}$ can be extended backward to a complete solution, closing smoothly over the singular orbit, then so can any sufficiently close end solution. This shows both that one can deform away from abelian solutions in a one-parameter family with $G_\infty>  \ell^{-1}$ fixed, and that the set of initial conditions $(f^+_1, g^+_1)$ leading to bounded solutions with $G_\infty > \ell^{-1}$ is open. 

Finally, we prove the existence of the solutions on the boundary $G_\infty = \ell^{-1}$ in \S \ref{sec:boundary} by matching initial solutions with the solutions at infinity, using a comparison argument from \S \ref{sec:existresBGGG}. In terms of initial conditions, this boundary corresponds to the solutions with $f^+_1= f(g^+_1)$. 

\subsection*{ALF-fibrations} \label{sec:alffibration} In the final section \S \ref{sec:taubnut}, we relate our two-parameter family of solutions of \eqref{eq:g2instantons00} to a two-parameter family of $SU(2) \times U(1)$-invariant ASD instantons on the ALF Taub-NUT metric on $\R^4$, fibred along the associative $S^3$. This geometry is realised by taking an adiabatic limit of the $\Bsev$-family, scaling the metric with different rates along the base $S^3$ and the normal fibre. Moreover, we give some seemingly new explicit formulae for these limiting ASD instantons in \S \ref{sec:taubnutasd}, by finding the general solution of the $SU(2) \times U(1)$-invariant ansatz considered in \cite{Kim2000169}, c.f. \cite[\S 10.3]{cherkis:2010}.

In \S \ref{sec:adiabatic}, we prove the following result, which can be found in more detail as Theorem \ref{theorem:asdconvergance}. 

\begin{thmx} \label{theorem:B} There are parameters $f^+_1(\ell), g^+_1(\ell) \rightarrow \infty$ as $\ell \rightarrow 0$ such that the two-parameter families of $G_2$-instantons in Theorem \ref{thm:A} converge uniformly on compact subsets of $S^3 \times \R^4$ to a two-parameter family of ASD instantons on Taub-NUT, transverse to a round $S^3$.   
\end{thmx}

Thanks to a recent general construction of anti-self dual instantons on ALF spaces in \cite{cherkisbows}, Theorem \ref{theorem:B} suggests another promising strategy for producing $G_2$-instantons on ALC spaces is to exploit the converse direction. That is, to lift the anti-self dual instantons produced by \cite{cherkisbows} to $G_2$-manifolds admitting such a fibration, to give $G_2$-instantons close to the adiabatic limit. In particular, since the moduli-spaces of ASD-instantons in \cite{cherkisbows} are hyperk\"{a}hler, this suggests that there are many non-invariant perturbations of the two-parameter family of instantons we have found. 

\subsection*{Acknowledgements} Special thanks to Lorenzo Foscolo, Jason Lotay, and Johannes Nordstr\"om for their helpful comments and discussions. The first author would also like to thank Derek Harland for drawing our attention to the reference \cite{etesi:2001}, and to Henrique N. S\'{a} Earp for his support of the project. Some of the formulae in \S \ref{sec:taubnutasd} were obtained by the first author in relation to a separate joint project with Lorenzo Foscolo and Calum Ross. The first author was funded by grant \#2023/02809-3, S\~{a}o Paulo Research Foundation (FAPESP), under the BRIDGES Collaboration \#2021/04065-6. The second author was funded by the EPSRC Studentship 2106787 and the Simons Collaboration on Special Holonomy in Geometry, Analysis and Physics \#488631.

% \subsection*{Conflicts of Interest, and Data Availability}
% The authors have no competing interests to declare that are relevant to the content of this article. No datasets were generated or analysed during the current study.

\subsection*{Statements and Declarations}
The authors have no competing interests to declare that are relevant to the content of this article. 
 
\section{$G_2$ structures}\label{section:G2define}

Following \cite{salamon:riemanniangeometry} c.f. \cite{ST23}, recall that a $G_2$-structure on a 7-manifold $M$ is a reduction of the frame bundle to the exceptional Lie group $G_2$. Existence of such a reduction is equivalent to the first and second Steifel-Whitney classes of $M$ vanishing, or to the existence of a non-degenerate $3$-form $\varphi \in \Omega^3 (M)$ which is fixed by the point-wise action of $G_2 \subset GL(7,\R)$ in some framing of the tangent space at each point. Furthermore, this data determines a Riemannian metric, and orientation on $M$, by the inclusion $G_2 \subset SO(7,\R)$. 

All the $G_2$-structures we consider in this paper are \textit{torsion-free}, i.e.  $d \varphi = d^* \varphi = 0$, where the co-differential $d^*$ is defined with respect to the induced metric. If the $G_2$-structure is torsion-free, then this induced metric is Ricci-flat \cite{salamon:riemanniangeometry}. 

Moreover, suppose $M$ is complete, and admits a co-homogeneity one action of a Lie group $G$. Let $K_0$ be the isotropy subgroup of the principal orbits. If the induced  Riemannian metric is irreducible, then it follows by Ricci-flatness there is a unique \textit{singular} $G$-orbit with isotropy subgroup $K \supset K_0$. In this paper, we will encode the co-homogeneity one action in a group diagram, namely 
\[K_0\subset K\subseteq G.\]

\subsection{The $\Bsev$ family}

We set out the construction of the $\Bsev$ family, as given in \cite{FHN18}. $M=S^3\times\R^4$ admits a 1-parameter family (up to scale) of complete $\SU(2)^2\times\text{U}(1)$-invariant ALC $G_2$-metrics, where $\SU(2)^2$ acts with co-homogeneity one, with the group diagram
\[\{1\}\subset\Delta\SU(2)\subset\SU(2)\times\SU(2).\]
Denote by $e_1, e_2, e_3$, $e'_1, e'_2, e'_3$ a basis of left-invariant one-forms on $S^3 \times S^3 = \SU(2)^2$ such that   
\begin{align*}
d e_i = - e_j \wedge e_k& &d e'_i = - e'_j \wedge e'_k
\end{align*}
for $\left(i j k\right)$ any cyclic permutation of $\left( 1 2 3 \right)$. The diagonal right action of $\SU(2)$ on the space of left-invariant one-forms is via two copies of the adjoint representation $\mathfrak{su}_+(2) \oplus  \mathfrak{su}_-(2)$, where $\mathfrak{su}_\pm(2)$ is given by the linear span of the one-forms $e_i^\pm := \tfrac{1}{2} \left(e_i \pm e'_i \right)$ over $i = 1,2,3$. 

Fix constants $p,q\in\R$. Then a co-homogeneity one $G_2$-structure on the space of principal orbits $\R_{>0} \times S^3 \times S^3$ of $M$ is the closed $\G$-invariant 3-form
\begin{align}\label{eqn:inv3-form}\varphi= pe_1\wedge e_2\wedge e_3 + q e_1'\wedge e_2'\wedge e_3' + d\,(a(e_1\wedge e_1'+ e_2\wedge e_2') + b e_3\wedge e_3'),\end{align}
where $d$ is the differential in 7 dimensions.

Let $F(a,b) := 4a^2 (b-p)(b+q) - (b^2 + pq)^2$. The requirement that $\varphi$ be co-closed with respect to the induced metric gives the Hitchin flow equations:
\begin{gather}
   \label{eq:hitchFHN} \dot{x}_1=\frac{\partial_aF(y_1,y_2)}{4\sqrt{F(y_1,y_2)}}, \quad \dot{x}_2=\frac{\partial_bF(y_1,y_2)}{2\sqrt{F(y_1,y_2)}},\quad
    \dot{y}_1=\frac{x_1x_2}{\sqrt{x_1^2x_2}}, \quad \dot{y}_2=\frac{x_1^2}{\sqrt{x_1^2x_2}}
\end{gather}
for the functions $x_1=\dot{a}\dot{b},\, x_2=\dot{a}^2,\, y_1=a$ and $y_2=b$.    

Proposition 4.5(i) of \cite{FHN18} gives the existence of a local 2-parameter family of solutions to \eqref{eq:hitchFHN}, i.e. torsion-free $G_2$-structures $\varphi$ defined in a neighbourhood of the singular orbit. These can be parameterised by a triple of real numbers $r_0$, $\bar{a}$, $\bar{b}$, subject to the constraints $64r_0(2\bar{a}+\bar{b})=1$, $p=-q=r_0^3$, $r_0>0$. Suppose $t$ is the arc-length parameter of a geodesic orthogonal to the principal orbits in $M$, then these local solutions have a power-series in $t$ given by
\begin{align} \label{eq:b7family}
a(t) = r_0^3+\frac{1}{4}r_0t^2+\bar{a}t^4+O(t^6)& &b(t)=r_0^3+\frac{1}{4}r_0t^2+\bar{b}t^4+O(t^6)
\end{align} 
near the singular orbit $S^3$ at $t=0$. 

Theorem 6.16(i) of \cite{FHN18} then extends these local solutions to complete global ALC $G_2$-metrics when $\bar{a}>\bar{b}$, and Theorem 6.16(ii) of \cite{FHN18} gives the known $\SU(2)^3$-invariant AC metric on the Bryant-Salamon $S^3\times\R^4$ when $\bar{a}=\bar{b}$. In other words, we have a complete family of solutions for $\tfrac{1}{3} \leq 64 \bar{a} r_0$, and fixing scaling $r_0$ results in the aforementioned 1-parameter family of complete $\SU(2)^2\times\text{U}(1)$-invariant ALC $G_2$-metrics when $\tfrac{1}{3} < 64 \bar{a} r_0$. 

For fixed scale, this family can also be parameterised by size of the asymptotic circle fibre, 
\begin{align*}
\ell:= \lim_{t \to \infty} \sqrt[3]{\tfrac{2 b^3}{ 3 a^2}}
\end{align*} 
with the asymptotic expansions for $a$, $b$ as $t \rightarrow \infty$ given by: 
\begin{align} \label{eq:b7familyasymp}
a(t) = \tfrac{1}{18}t^{3} + O(t^2),& &b(t)=\tfrac{\ell}{6}t^{2}+ O(t).
\end{align} 
The limit $\ell \rightarrow 0$ as the size of the circle fibre collapses appears here as the limit $\bar{a}\rightarrow   \infty$. 

We have some useful inequalities for the functions $a$ and $b$:
\begin{lemma} Let $a,b$ be the solutions of \eqref{eq:hitchFHN} given by the $\Bsev$-family \eqref{eq:b7family}, parameterised by geodesic arc-length $t$. Then for all $t>0$, we have:
\begin{align} \label{eq:b7ineq}
b>p>0,& &\dot{b}>0,& &\frac{\ddot{a}}{\ddot{b}}>\frac{\dot{a}}{\dot{b}}>\frac{a}{b}>1.
\end{align}
\end{lemma}
\begin{proof} The first two inequalities follow from the requirement that, if $t$ is the arc-length parameter, then $F(a,b)=4\dot{a}^4\dot{b}^2$. The inequalities $a>b$, $\dot{a}b> \dot{b}a$ are shown in \cite[Lemma 6.4]{FHN18}, \cite[Proposition 6.11]{FHN18} respectively. The remaining inequality follows by showing that $\tfrac{\dot{a}}{\dot{b}}$ is strictly increasing: this follows implicitly from the proof of \cite[Proposition 7.6]{FHN18}.   
\end{proof}

\subsubsection{Metric equations}

As it will be useful in \S \ref{sec:taubnut}, and to facilitate the comparison with the ODEs in \cite{LO18b}, we also consider another parametrisation of the system \eqref{eq:hitchFHN}. 

Write  functions $A_i,B_i:\R_+\to\R$ given by

\begin{align}\label{eq:LOtoFHN}
A_1=A_2 = \sqrt{\frac{(b-p)(2a-b-p)}{\dot{a}\dot{b}}},\hspace{1cm} & B_1=B_2 = \sqrt{\frac{(b-p)(2a+b+p)}{\dot{a}\dot{b}}}\\
A_3 = \frac{b-p}{\dot{a}},\hspace{1cm} & B_3 = \frac{2\dot{a}\dot{b}}{b-p} .\nonumber
\end{align}

\begin{comment}
\begin{align}
&&\frac{1}{4} \left( A_1^2 + B_1^2 \right)&=\frac{a \left( b- p \right)}{ \dot{a} \dot{b}}&&
&\frac{1}{4} \left( A_1^2 - B_1^2 \right)&=\frac{-b^2 + p^2}{ 2 \dot{a} \dot{b}}&& \\
&&\frac{1}{4} \left( A_3^2 + B_3^2 \right)&=\frac{a^2 - p b }{ \dot{a}^2}&&
& \frac{1}{4} \left( A_3^2 - B_3^2 \right)&=\frac{b^2 - 2a^2 + p^2 }{ 2 \dot{a}^2}&&\nonumber
\end{align}

\begin{align}
\label{eqn:metric}
g_t & =\frac{a(b-p)}{\dot{a}\dot{b}}(e_1\otimes e_1 + e_2\otimes e_2) + \frac{a^2-bp}{\dot{a}^2}e_3\otimes e_3
 + \frac{a(b+q)}{\dot{a}\dot{b}}(e_1'\otimes e_1' + e_2'\otimes e_2') + \frac{a^2+bq}{\dot{a}^2}e_3'\otimes e_3'\\
&\hspace{1cm} - \frac{b^2+pq}{2\dot{a}\dot{b}}(e_1\otimes e_1' + e_1'\otimes e_1 + e_2\otimes e_2' + e_2'\otimes e_2) + \frac{b^2-2a^2-pq}{2\dot{a}^2}(e_3\otimes e_3' + e_3'\otimes e_3).\nonumber
\end{align}
\end{comment}

Let $t$ be the arc-length parameter of a geodesic orthogonal to the principal orbits, and denote by $g_t$ the metric on $S^3 \times S^3$ induced by the inclusion into $M$ as the orbit $\{t\}\times S^3 \times S^3$, so that the metric on $M$ appears as $g=dt^2+g_t$. Madsen and Salamon \cite[Chapter 5]{MS12} give the following explicit formula for $g_t$:
\begin{align}
\label{eqn:metric}
g_t & = \sum_{i=1}^3 \left[ A_i^2 \left(e_i^+ \right)^2 + B_i^2 \left(e_i^- \right)^2\right].
\end{align}

The Hitchin flow equations \eqref{eq:hitchFHN} transform to ODEs for the functions $A_i,B_i$, namely
\begin{gather} \label{eq:hitchFHNmetric}
\begin{aligned}
    \dot{A}_1 & = \frac{1}{2}\left(\frac{B_1^2+B_3^2-A_1^2}{B_1B_3}-\frac{A_3}{A_1}\right), \\
    \dot{A}_3 & = \frac{1}{2}\left(\frac{A_3^2}{A_1^2}-\frac{A_3^2}{B_1^2}\right),\\
    \dot{B}_1 & = \frac{1}{2}\left(\frac{A_1^2+B_3^2-B_1^2}{A_1B_3}+\frac{A_3}{B_1}\right), \\
    \dot{B}_3 & = \frac{A_1^2+B_1^2-B_3^2}{A_1B_1}.
\end{aligned}
\end{gather}

Note that the parameter $\ell$ appears as the limit $\ell = \lim_{t \rightarrow \infty} A_3 $, and that the freedom to re-scale the resulting metric, up to re-parameterising $t \mapsto \lambda t$, appears as the symmetry 
\begin{align*}
 A_i(t) \mapsto \lambda ^{-1} A_i ( \lambda t)& & B_i(t) \mapsto \lambda ^{-1} B_i ( \lambda t )   
\end{align*}
of the ODE system \eqref{eq:hitchFHNmetric}. 

Finally, we denote $E_i$, $E_i'$, for $i =1,2,3$, as the dual basis of left-invariant vector-fields on $S^3 \times S^3$, such that $e_i(E_i) = e_i'(E_i') =1$. Up to scale, we can identify the vector-field generated by the additional $U(1)$-action on the $G_2$-structure \eqref{eqn:inv3-form} with $E_3 + E_3'$, and its metric dual $A_3^2 e^+_3$. We will see in the next section that this harmonic one-form gives a solution of \eqref{eq:g2instantons00}.

\section{$G_2$-instantons}
\label{section:G2overview} 

We recall some of the basic formula from \cite{LO18b} c.f. \cite{ST23}, used to derive the ODEs corresponding to the $G_2$-instanton equations in the following sections. First, let $\left(M^7,\varphi \right)$ be a $G_2$-manifold, equipped with a principal $G$-bundle $P \rightarrow M$ for a compact, semi-simple Lie group $G$. Recall that a connection $A$ on $P$ is called a $G_2$-\textit{instanton} if it satisfies the $G_2$-\textit{instanton equations}: 
\begin{align} \label{eq:g2instanton0}
F_A \wedge {*} \varphi = 0. 
\end{align}
Suppose that $M$ is co-homogeneity one, with principal orbit $N$, and identify the space of principal orbits with $\R_{>0} \times N$. Let $\R_{>0}$ be parameterised by arc-length $t$ of an orthogonal geodesic to $N$, so that we can write the $G_2$-structure on the space of principal orbits as: 
\begin{align*} 
\varphi = dt \wedge \omega + \Reom& &*\varphi = - dt \wedge \Imom + \tfrac{1}{2} \omega^2
\end{align*} 
for $(\omega, \Reom, \Imom)$ the $t$-dependant family of $SU(3)$-structures on $N$ defined by:
\begin{align} \label{su3define}
\omega = \iota^* \left( \partial_t \lrcorner\ \varphi \right)& &\Reom = \iota^* \varphi& &\Imom = \iota^* \left(- \partial_t \lrcorner *\varphi \right)
\end{align}
where $\iota:N \hookrightarrow M$ is the inclusion of $N$ into the space of principal orbits as $\lbrace t \rbrace \times N$.

 Written in a gauge such that $A=A_t$, the $G_2$-instanton equations \eqref{eq:g2instanton0} can be written on $N$ as the evolution equation:
\begin{subequations} \label{g2instanton}
\begin{align} 
F_{A_t} \wedge \omega^2 =0, \label{g2inst1}\\
F_{A_t} \wedge \Imom - \frac{1}{2} \partial_t A_t \wedge \omega^2 = 0. \label{g2inst2} 
\end{align}
\end{subequations}
As is shown in \cite{ST23}, the $G_2$-instanton equation \eqref{g2inst1} is preserved under evolution by \eqref{g2inst2}.  
\subsection{$G_2$-instanton Equations}
Now suppose that $M = S^3 \times \R^4$ is a member of the $\Bsev$ family: in particular, $M$ admits a co-homogeneity one action of $\G$ with principal orbits $S^3\times S^3$ and singular orbit $(S^3\times S^3)/\Delta S^3$. We now describe the possible $\G$-homogeneous $\SU(2)$-bundles on the orbits of the $\Bsev$ family. The additional U(1)-symmetry can be encoded by writing the principal orbits in the form
\[\SU(2)^2\cong\SU(2)^2\times\text{U}(1)/\Delta\text{U}(1).\]
Isomorphism classes of homogeneous $\SU(2)$-bundles on the principal orbits are in correspondence with conjugacy classes of isotropy homomorphisms $\lambda_p:\Delta\text{U}(1)\to \SU(2)$. Such classes are parameterised by $k\in\mathbb{Z}$ where
\[\lambda_p^k: e^{i\theta}\mapsto \left(\begin{array}{cc} e^{ik\theta} & 0 \\ 0 & e^{-ik\theta}\end{array}\right).\]

The singular orbit can be written in the form
\[\SU(2)^2/\Delta\SU(2)\cong (\SU(2)^2\times\text{U}(1))/(\Delta\SU(2)\times\text{U}(1)).\]
In this case, the possible isotropy homomorphisms $\Delta\SU(2)\times\text{U}(1)\to\SU(2)$ are 
\[\lambda_s^{1,k}:(g,e^{i\theta})\mapsto\left(\begin{array}{cc} e^{ik\theta} & 0 \\ 0 & e^{-ik\theta}\end{array}\right)\]
for $k\in\mathbb{Z}$, and
\[\lambda_s^{\text{id},0}:(g,e^{i\theta})\mapsto g.\]

Take the complement of $\Delta\su(2)\oplus\:\!\mathfrak{u}(1)$ to be $\mathfrak{m}=\Delta^-\su(2) \oplus\:\! 0$, where $\Delta^-\su(2)$ is the anti-diagonal copy of $\su(2)$ in $\su(2)\oplus\su(2)$. Any invariant connection on $P_s^{1,k}$ can be written as $d\lambda_s^{1,k}+\Lambda_{1,k}$, where $d\lambda_s^{1,k}=E_3\otimes kd\theta\in \su(2)\otimes T^*_eM$ is the canonical invariant connection and $\Lambda_{1,k}\in\su(2)\otimes T^*_eM$ is given by a morphism of $\Delta\SU(2)\times\text{U}(1)$-representations
\[\Lambda_{1,k}:(\mathfrak{m},\text{Ad})\to (\su(2),\text{Ad}\circ\lambda_s^{1,k}).\]
Indeed, we can extend $d\lambda_s^{1,k}+\Lambda_{1,k}$ to an $\Ad(P_s^{1,k})$-valued 1-form by left invariance; the property that $\Lambda_{1,k}$ is a morphism of representations ensures that the values lie in the adjoint bundle. Considering the splitting into irreducibles, Schur's Lemma tells us that $k=1$ for non-abelian connections and hence there are two bundles $P_1$ and $\Pid$, which are two different extensions over the singular orbit of the unique bundle on $\R^+\times\SU(2)^2$.

Following \cite{LO18b}, we can write an $SU(2)^2 \times U(1)$-invariant connection on $P_1$ or $\Pid$ over $M$ as
\begin{align}\label{eq:AB7}A= A_1f^+\Big(\sum_{i=1,2}E_i\otimes e_i^+\Big)+B_1f^-\Big(\sum_{i=1,2}E_i\otimes e_i^-\Big)+A_3g^+ E_3\otimes e_3^++B_3g^- E_3\otimes e_3^-.\end{align}
with functions $f^{\pm},g^{\pm}:\R^+\to\R$. Then the $G_2$-instanton equations \eqref{g2instanton} transform to
\begin{align}\label{eq:B7odes}
    \dot{f}^+\ \ =\ \ &
    \frac{\dot{a}(4a^2 -(b+p)^2-2a(b-p))-\dot{b}(a-b)(2a+b+p)}{(b-p)(4a^2-(b+p)^2)}f^++f^-g^--f^+g^+, \nonumber \\[0.8ex]
    \dot{f}^-\ \ =\ \ &
     \frac{\dot{a}(4a^2 -(b+p)^2+2a(b-p))+\dot{b}(a+b)(2a-b-p)}{(b-p)(4a^2-(b+p)^2)}f^-+f^+g^-+f^-g^+, \\[0.8ex]
    \dot{g}^+\ \ =\ \ &  \frac{\dot{b}(b+p)}{(4a^2-(b+p)^2)}g^++(f^-)^2-(f^+)^2,\nonumber \\[0.8ex] 
    \dot{g}^-\ \ =\ \ &-\frac{4a\dot{a}(b+p)+\dot{b}(4a^2-(b+p)^2)}{(b-p)(4a^2-(b+p)^2)}g^-+2f^+f^-.\nonumber
\end{align}

\begin{remark}
A more general set of instanton ODEs for the $SU(2)^2 \times U(1)$-invariant $G_2$-metrics of \cite{FHN18} can be found in \cite{MNT22}.
\end{remark}

\begin{remark}
Note: due to different conventions for the choice of a U(1) basis generator, the functions $f^\pm$ and  $g^\pm$ switch roles in the notation of \cite{LO18b}.
\end{remark}

The following proposition of \cite{LO18b} proves that on $P_1$, we need only consider the case where $f^-=g^-=0$.

\begin{proposition}[{\cite[Proposition 9]{LO18b}}]\label{prop:9}
Let $X\subset S^3\times\R^4$ contain the singular orbit $S^3\times\{0\}$ of the $\G$-action and be equipped with an $\G$-invariant $G_2$-metric. There is a 2-parameter family of $\G$-invariant $G_2$-instantons with gauge group $\SU(2)$ in a neighbourhood of the singular orbit in $X$ extending smoothly over $P_1$. 

Moreover, any such $G_2$-instanton can be written as in \eqref{eq:AB7} with $f^-=g^-=0$, and $f^+$, $g^+$ solving the ODEs
\begin{align}\label{eq:BGGGodes}
    \dot{f}^++\frac{1}{2}\left(\frac{A_1^2+B_3^2+B_1^2}{A_1B_1B_3}-\frac{A_3^2+2A_1^2}{A_1^2A_3}\right)f^+ & =-f^+g^+,\\
    \dot{g}^++\frac{1}{2}\left(\frac{A_3}{B_1^2}-\frac{A_3}{A_1^2}\right)g^+ & =-(f^+)^2, \nonumber
\end{align}
where $f^+=f_1^+t+O(t^3)$ and $g^+=g_1^+t+O(t^3)$ for parameters  $f_1^+,g_1^+\in\R$.
\end{proposition}

For the remainder of this section, we will classify $SU(2)^2 \times U(1)$-invariant $G_2$-instantons on $P_1$ with bounded curvature, by determining which of the two-parameter family $(f_1^+,g_1^+)$ of local solutions to \eqref{eq:BGGGodes} in Proposition \ref{prop:9} are bounded and defined for all forward time. We describe this subset of initial conditions as a subset of $\R^2$, and analyse properties of the corresponding solutions, including their asymptotics.

\subsection{Extending results of Lotay-Oliveira}\label{sec:existresBGGG} We begin our analysis with the low-hanging fruit: describing the results of \cite{LO18b} which extend to the whole $\Bsev$-family without much additional work, but are necessary to prove the main theorem. 

Firstly, the one-parameter family of local solutions with $f_1^+=0$ extend over the whole manifold, as reducible abelian $G_2$-instantons on $P_1$ with $f^+=0$. 
\begin{proposition}[{\cite[Proposition 2]{LO18b}}]\label{lem:abBGGG}
Abelian solutions of \eqref{eq:BGGGodes} are in a one-parameter family, extending smoothly over the singular orbit of $P_1$. These are given explicitly by: 
\begin{align} \label{eq:abelian}
f^+ =0,& &g^+ = 2g_1^+ A_3.
\end{align}
\end{proposition}
Note also that, in particular, the critical point $f^+ = g^+ = 0$ of \eqref{eq:BGGGodes} defines a flat connection. Now, since \eqref{eq:BGGGodes} defines a smooth initial-value problem away from $t=0$, we have the following corollary to Proposition \ref{lem:abBGGG}: 
\begin{corollary}[{\cite[Lemma 6]{LO18b}}]\label{lem:6}
Suppose $(f^+, g^+)$ is a solution to \eqref{eq:BGGGodes}. If the inequality $f^+(t_0)>0$ holds for some $t_0>0$, it holds for all $t>0$ for which the solution exists.
\end{corollary}
Combined with the symmetry  $(f^+,g^+) \mapsto (-f^+,g^+) $ of \eqref{eq:BGGGodes}, coming from the adjoint action of imaginary number $i \in SU(2)$ on the connection \eqref{eq:AB7} by gauge transformations, Corollary \ref{lem:6} ensures that we can always restrict to the case $f^+\geq 0$ in our analysis.  

Finally, there is a unbounded region of initial conditions for which the corresponding local solutions do not extend to complete bounded solutions on $M$.

\begin{theorem}[{\cite[Theorem 8]{LO18b}}]\label{thm:8}
Let $A$ be a $\G$-invariant $G_2$-instanton on $P_1$ defined in a neighbourhood of the singular orbit of the $\Bsev$-family as given by Proposition \ref{prop:9}. If $g^+_1 \leq \frac{1}{2}\ell^{-2}$, or $f^+_1 \geq 0$ with $f^+_1 \geq g^+_1$, then $A$ extends globally to $M$ with bounded curvature if and only if $f^+_1=0$.
\end{theorem}

We will not repeat the full details of the proof of Theorem \ref{thm:8}, as it is unchanged from \cite{LO18b}. The only additional result needed is to show a certain smooth function $H$ of $t \in [0,\infty )$ is strictly positive for the whole $\Bsev$-family, c.f. \cite[Remark 12]{LO18b}. This function is given explicitly by:   
\begin{gather} \label{eq:H}
\begin{aligned}
    H & =\frac{(2\dot{a}^2-\dot{a}\dot{b})(4a^2-(b+p)^2)+\dot{a}\dot{b}(b-p)(2a-b-p)-4\dot{a}^2a(b-p)}{2(b-p)^2(4a^2-(b+p)^2)}\\
    & = \frac{1}{2}\left(\frac{2}{A_3^2}+\frac{1}{B_1^2}-\frac{A_1^2+B_1^2+B_3^2}{A_1A_3B_1B_3}\right).
\end{aligned}
\end{gather}
% We will not re-prove this theorem in full here, but we will make some remarks on how to generalise the proof i. It is proved by setting:
% \[\frac{ds}{dt}=A_3=\frac{b-p}{\dot{a}}\]
% defining
% \[F=A_3^{-1}f^+ \ \ \ \text{ and } \ \ \ G=A_3^{-1}g^+.\]
% Then, writing $F'=\frac{dF}{ds}$, we have
% \[F'=(H-G)F \ \ \ \text{ and } G'=-F^2\]
% where , given explicitely by    
% \begin{remark}
% Near the singular orbit, we have
% \begin{align*}
% H = \tfrac{192 r_0 \bar{a} - 1}{2 r_0^2} +  \tfrac{\left(\bar{a} - \tfrac{1}{192 r_0}\right)\left(\bar{a} - \tfrac{1}{160 r_0}\right)}{ r_0^2}t^2 + O(t^6)
% \end{align*}
%  Recall that $\bar{a} \geq \tfrac{1}{192 r_0}$ with equality iff $a=b$. Thus $\dot{H}<0$ near $0$ if $ \tfrac{1}{160 r_0}> \bar{a}$
% \end{remark} 
To do this, we re-write $H  =\frac{\dot{a}^2}{2(b-p)^2(4a^2-(b+p)^2)} \hat{H}$, where $\hat{H}$ is defined in terms of $a,b,p, b' := \tfrac{\dot{b}}{\dot{a}}$, as:
\[\hat{H}:= (2-b')(4a^2-(b+p)^2 - 2a (b-p))- b' (b-p)(b+p). \]
We recall some basic inequalities from \S \ref{section:G2define}: $1>b'>0$ and $a>b>p > 0$. From these, it follows immediately that $(b-p)(b+p)>0$, and $4a^2-(b+p)^2 - 2a (b-p)> 0$ by expanding out the terms. Now, since these terms are both positive, it follows from the upper bound on $b'$ that $\hat{H}> 0$, and thus $H>0$.

%It is worth noting that the formula involving $r$ in (4.23) of \cite{LO18b} is not quite right. The fraction part should be divided by 2. Then $H(0)=5/9$ for the BGGG metric if we assume $\ell=1$ as a fix of the scaling and $r_0=\frac{3}{4}$. Then the metric is in the red part of Figure \ref{fig:flow} as we would expect (it is in the purple part when you calculate it with the original formula).

\begin{comment}
$G_2$-instantons on ALC $G_2$-manifolds are asymptotic to a U(1)-invariant connection 
\[A_{\infty}=a+\ell\Phi\otimes\theta\]
on a principal bundle $P$ over $(1,\infty)\times\Sigma$, and that $(a,\Phi)$ is a CY monopole on the CY cone $(1,\infty)\times X$. The BGGG metric has $\ell=1$ and $\theta=2\eta^+_3$. The mass of the monopole is
\[m=\lim_{t\to\infty}\vert\Phi\vert\]
so writing $A$ as in \eqref{eq:AB7}, we have
\[m=\lim_{t\to\infty}\left\vert\frac{A_3}{2}g^+T_3\right\vert=\frac{G_{\infty}}{2}.\]
\end{comment}

\subsection{Properties of solutions}\label{sec:propofsols}

In order to understand solutions to the $G_2$-instanton equations \eqref{eq:BGGGodes} and prove the main theorem, we want to study their asymptotic behaviour. To that end, we note using \S \ref{section:G2define}, we can write the coefficients in \eqref{eq:BGGGodes} as:
\begin{equation}\label{eq:B7eqcoeff}\frac{1}{2}\left(\frac{A_3}{B_1^2}-\frac{A_3}{A_1^2}\right) = - \frac{\dot{A}_3}{A_3} = \frac{-\dot{b}(b+p)}{4a^2-(b+p)^2}<0\end{equation}
which converges to 0 as $t\rightarrow \infty$, and; 
\[  \frac{1}{2} \left(\frac{A_1^2+B_3^2+B_1^2}{A_1B_1B_3}-\frac{A_3^2+2A_1^2}{A_1^2A_3} \right) = -A_3\left(\frac{\dot{A}_3}{A_3^2}+H\right)<0\]
which tends to $-\ell^{-1}$ as $t\to\infty$. Moreover, using \eqref{eq:b7familyasymp}, we have the asymptotic expansions:

\begin{align}\label{eq:BGGGasym}
    \frac{1}{2}\left(\frac{A_3}{B_1^2}-\frac{A_3}{A_1^2}\right) & = -\frac{9}{2}\ell^2t^{-3}+\gamma_1t^{-4},\\
    \frac{1}{2}\left(\frac{A_1^2+B_3^2+B_1^2}{A_1B_1B_3}-\frac{A_3^2+2A_1^2}{A_1^2A_3}\right) & = -\ell^{-1}+\frac{5}{2}t^{-1}+\gamma_2t^{-2},\nonumber
\end{align}
where $\gamma_1, \gamma_2$ are bounded functions.
\begin{comment}
converging to non-zero constants to $O(t^{-1})$
 \begin{align*}
    \gamma_1 & \sim \frac{27\sqrt{3}}{2}\rho\ell^2+O(t^{-1}),\\
    \gamma_2 & \sim -\frac{15}{4}\left(\frac{2}{\sqrt{3}}\rho+\frac{3}{2}\ell\right)+O(t^{-1}),
\end{align*}
for some parameter $\rho$.    
\end{comment}
\begin{lemma} \label{lem:bounded}
If $g^+(t_0)<0$ for some $t_0>0$, then $g^+(t)<0$ for all $t\geq t_0$. Moreover, if $(f^+, g^+)$ is a complete, bounded solution of \eqref{eq:BGGGodes}, then $G_{\infty}=\lim_{t\to\infty}g^+$ exists. If also $f^+ \neq 0$, then $G_{\infty}\geq \ell^{-1}$, $\lim_{t\to\infty}f^+=0$, and $g^+>0$ for all $t>0$. 
\label{lem:posf}
\end{lemma}
\begin{proof}
Suppose there exists some $t_0>0$ such that $g^+(t_0)\leq0$, then
\[\dot{g}^+(t_0)=-(f^+)^2(t_0)-\frac{1}{2}\left(\frac{A_3}{B_1^2}-\frac{A_3}{A_1^2}\right)g^+(t_0)\leq0.\] 
Then $g^+$ is decreasing and so stays non-positive for all $t\geq t_0$. 

Now, we can re-write the second equation of \eqref{eq:BGGGodes} as  
\begin{align} \label{eq:expBGGG}
\frac{d}{dt} \left( \frac{g^+}{A_3} \right) = - \frac{(f^+)^2}{A_3}
\end{align}
Thus by monotonicity, if $g^+$ is bounded, then it must have a limit $G_\infty$. Since \eqref{eq:expBGGG} then implies $f^+$ must be square-integrable on $t\geq t_0$, and the first equation of \eqref{eq:BGGGodes} implies it has bounded derivative, we must also have $f^+ \rightarrow 0$ in the limit. 

The first equation also gives
\begin{equation}\label{eq:derlogg}\frac{d}{dt}\log f^+=-\frac{1}{2}\left(\frac{A_1^2+B_3^2+B_1^2}{A_1B_1B_3}-\frac{A_3^2+2A_1^2}{A_1^2A_3}\right)-g^+ \to \ell^{-1}-G_{\infty};\end{equation}
so then $f^+$ is bounded in the limit only if $\ell^{-1}-G_{\infty}\leq 0$.
\end{proof}

In addition to the first-order asymptotic behaviour described by the previous lemma, the following lemma describes the decay rate of the function $f^+$ for a complete bounded solution $(f^+,g^+)$.

\begin{lemma} \label{lem:expdecay}
Suppose $(f^+,g^+)$ is a complete bounded solution of \eqref{eq:BGGGodes}. We can write
\begin{equation}\label{eq:g+}
    f^+=\exp((\ell^{-1}-G_{\infty})t)t^{-\frac{5}{2}}h(t)
\end{equation}
for some bounded function $h$ satisfying
\begin{equation}\label{eq:h}
\dot{h}=h(G_{\infty}-g^+-\gamma_2t^{-2})=-h(k+\gamma_2)t^{-2}.
\end{equation}
Here, $\gamma_2$ is defined in \eqref{eq:BGGGasym} and we define $k$ to be the bounded function satisfying
\[g^+=G_{\infty}+t^{-2}k(t).\]
\end{lemma}
\begin{proof}
From \eqref{eq:expBGGG}, we have that
\begin{align} \label{eq:gexpansion}
g^+(t)= A_3 \left( G_{\infty} \ell^{-1} + \int_t^{\infty}\tfrac{1}{A_3}(f^+(t))^2dt \right).   
\end{align}
The second term must vanish asymptotically since  $G_{\infty}=\lim_{t\to\infty}g^+$, thus we see that a series expansion of $g^+$ for large $t$ is of the form, for some bounded function $k$,
\begin{equation}\label{eq:f+}g^+=G_{\infty}+t^{-2}k(t).\end{equation}

Substituting this expansion into equation \eqref{eq:derlogg} and integrating yields
\[\log f^+=(\ell^{-1}-G_{\infty})t-\tfrac{5}{2}\log t+O(t^{-1}))\]
and hence we can write
\begin{equation*}
    f^+=\exp((\ell^{-1}-G_{\infty})t)t^{-\frac{5}{2}}h(t)
\end{equation*}
for some function $h$ satisfying
\begin{equation*}
\dot{h}=h(G_{\infty}-g^+-\gamma_2t^{-2})=-h(k+\gamma_2)t^{-2}.
\end{equation*}
Since $\frac{d}{dt}\log h$ is $O(t^{-2})$, hence integrable, $\log h$ decays as $t\to\infty$. Then $h$ is bounded for sufficiently large $t$.
\end{proof}

Now suppose $(f^+,g^+)$ is a solution to \eqref{eq:BGGGodes} such that $g^+\to G_{\infty}\geq \ell^{-1}$ as $t\to\infty$. Let
\[g^+_{\text{ab}}=G_{\infty}\ell^{-1} A_3,\ \ \ \ \ \ \  f^+_{\text{ab}}=0\]
be the abelian solution of Proposition \ref{lem:abBGGG} with $g^+_{\text{ab}}\to G_{\infty}$. Plugging this into \eqref{eq:gexpansion}, and using \eqref{eq:g+} on some interval $(T,\infty)$ for sufficiently large $T\in\R$, yields
\begin{equation} \label{eq:f+decay}
g^+=g^+_{\text{ab}} + o(t^{-\frac{5}{2}}\exp((\ell^{-1}-G_{\infty})t)).
\end{equation}

We now prove a couple of lemmas which compare two solutions of \eqref{eq:BGGGodes}. These lemmas will be useful when we consider a sequence of solutions whose initial conditions move in the direction of the incomplete set of Theorem \ref{thm:ic}. 
 
\begin{lemma}
\label{lem:botright}
Suppose $(\hat{f}^+,\hat{g}^+)$,  $(f^+,g^+)$ are two distinct solutions of \eqref{eq:BGGGodes} such that \[g^+\geq\hat{g}^+\hspace{1cm} \text{ and }\hspace{1cm} \hat{f}^+\geq f^+>0.\] at time $t_0>0$. Then $g^+>\hat{g}^+$ and $\hat{f}^+>f^+$ for all $t> t_0$. 
\end{lemma}
\begin{proof}
Suppose that $g^+=\hat{g}^+$, and $\hat{f}^+>f^+>0$ at some time $t_1\geq t_0$. Then
\[\frac{d}{dt}(g^+-\hat{g}^+)(t_1)=(\hat{f}^+)^2(t_1)-(f^+)^2(t_1)>0.\]
On the other hand, if we have  $g^+>\hat{g}^+$, and $\hat{f}^+ = f^+>0$ at time $t_1 \geq t_0$, 
\[\frac{d}{dt}(\hat{f}^+ - f^+)(t_1)= \hat{f}^+(t_1) ( g^+(t_1)-\hat{g}^+(t_1))>0.\]
By uniqueness of solutions, there cannot be a time $t_1>t_0$ such that the inequalities $g^+>\hat{g}^+$ and $\hat{f}^+>f^+$ fail simultaneously, thus the lemma follows. 
\end{proof}

\begin{remark} The conclusion in Lemma \ref{lem:botright} also holds for distinct initial conditions satisfying the  inequalities $g_1^+\geq\hat{g}_1^+$, $\hat{f}_1^+\geq f_1^+>0$. This follows by observing that if $f_1^+=\hat{f}_1^+$, $g_1^+\neq\hat{g}_1^+$ then we have the following expansion for the corresponding solutions near $t=0$:
\begin{align*}
    f^+-\hat{f}^+= \hat{f}^+_1(\hat{g}^+_1-g^+_1) t^3 + O(t^5), 
\end{align*}
 while if instead $g_1^+=\hat{g}_1^+$, $f_1^+ \neq \hat{f}_1^+$ we have \begin{align*}
    g^+-\hat{g}^+= \tfrac{1}{2}((\hat{f}^+_1)^2-(f^+_1)^2) t^3 + O(t^5).
\end{align*}
\end{remark}

We will now prove that any solution $(f^+,g^+)$ with an initial condition satisfying the assumptions of Lemma \ref{lem:botright} for a fixed complete solution $(\hat{f}^+,\hat{g}^+)$ is itself convergent, and the limit points of these two solutions are distinct.

\begin{lemma}\label{lem:orderoflims}  
With the assumptions of Lemma \ref{lem:botright}, suppose that $(\hat{f}^+, \hat{g}^+)$ is a complete solution with $\hat{g}^+ \to \hat{G}_{\infty}$, then $(f^+,g^+)$ is also complete, and also converges with $g^+ \rightarrow G_{\infty} > \hat{G}_{\infty}$. 
\end{lemma}
\begin{proof}
Lemma \ref{lem:botright} tells us that if we start at a complete solution anywhere in the space of positive initial conditions, then choosing another set of initial conditions satisfying the inequalities of Lemma \ref{lem:botright}, we still have a complete solution with $g^+$ tending to a finite limit $G_{\infty}\geq\ell^{-1}$. We can see this by supposing that the solution $(\hat{f}^+,\hat{g}^+)$ is complete while the solution $(f^+,g^+)$ is not. Lemma \ref{lem:botright} tells us that $g^+$ must go to positive infinity in the limit. However, this is impossible, since we have that $\tfrac{g^+}{A_3}$ is strictly decreasing by \eqref{eq:expBGGG}. 

Clearly we have $\hat{G}_{\infty}\leq G_{\infty}$. Now suppose that $\hat{G}_{\infty}=G_{\infty}$; then the function $g^+-\hat{g}^+$ must be a positive function that tends to zero. From \eqref{eq:B7eqcoeff}, we have
\[\frac{d}{dt}(g^+-\hat{g}^+)=-\frac{1}{2}\left(\frac{A_3}{B_1^2}-\frac{A_3}{A_1^2}\right)(g^+-\hat{g}^+)+(\hat{f}^+-f^+)(\hat{f}^++f^+)>0\]
for all $t>0$ and so the function $g^+-\hat{g}^+$ is an increasing function. This contradicts the assumption that it tends to zero. Thus the limits must be distinct.
\end{proof}

We now state the main result which fully characterises the asymptotic behaviour of all complete bounded solutions to \eqref{eq:BGGGodes}. Recall that solutions $(f^+,g^+)$ have initial conditions $(f^+_1,g^+_1)$. 

\begin{theorem}\label{prop:bndry}
The region of initial conditions which lead to complete bounded solutions is closed. It is the union of $\{f_1^+=0,g_1^+<\frac{1}{2}\ell^{-2}\}$ and a closed set with boundary given by a graph of an increasing continuous function $g_1^+\mapsto f^+_1(g_1^+)$ and its reflection in the $g_1^+$-axis. The boundary passes through the point $(f_1^+,g_1^+)=(0,\frac{1}{2}\ell^{-2})$ corresponding to the abelian solution with $f=0, G_{\infty}=\ell^{-1}$, and corresponds to solutions with $\lim_{t\to\infty}g^+=G_{\infty}=\ell^{-1}$.
\end{theorem}
The proof of this theorem involves a number of steps which we now outline. 
\begin{itemize}
    \item We start by parameterising local solutions near the singular orbit by writing $s=t^{-1}$; Proposition \ref{prop:ends} fully characterises this family of solutions, namely that $g^+\to G_{\infty}$ and $f^+$ decays exponentially with rate $\ell^{-1}-G_{\infty}$ as $t\to\infty$, and the family depends continuously on both parameters.
    \item We then show that we can perturb abelian solutions with $f_1^+= 0$, $g_1^+>\frac{1}{2}\ell^{-2}$ in a two-parameter family, by extending nearby solutions in Proposition \ref{prop:ends} backward in $t$, until they extend smoothly over the singular orbit at $t=0$.
    \item In Section \ref{sec:propofsols}, we showed that as we vary the initial conditions $f_1^+$ and $g_1^+$ away from the one-parameter family of abelian solutions, the solutions satisfy inequalities outlined in Lemma \ref{lem:botright} for all time, including in the limit as $t\to\infty$. This allows us to construct a monotonic sequence of complete solutions converging to any initial conditions in the set of complete solutions. The sequence corresponds to solutions, each of which have an asymptotic limit $G_{\infty}$ for $g^+$. Hence, we also have a sequence of real numbers $(G_{\infty})_n$ that converges to some $G_{\infty}\geq\ell^{-1}$.  
    
    \item Finally we consider two solutions, the first corresponding to the limit of the sequence of initial conditions described above, and the second given by the limit of the sequence $(G_{\infty})_n$. We use the continuity result of Proposition \ref{prop:ends} to show that these solutions coincide globally. This global solution is the limit of the aforementioned sequence of solutions, and hence we show that the set of complete bounded solutions is closed, with a boundary that corresponds to solutions with $G_{\infty}=\ell^{-1}$, as shown in Figure \ref{fig:boundary}.
\end{itemize}

\begin{remark} \label{remark:LO} In  \cite[Theorem 9]{LO18b}, the existence of the perturbation in the second step can be read off from the initial conditions alone, as their argument shows that initial conditions $(f^+_1,g^+_1)$ with $g^+_1 \geq \tfrac{1}{2}\sup H +f^+_1 > \tfrac{1}{2}\sup H$ lead to complete solutions, where $H$ is given in \eqref{eq:H}, c.f. \cite[Remark 13]{LO18b}. For the explicit BBBG metric considered in \cite{LO18b}, $H$ is monotonically increasing, thus $\sup H = \lim_{t\rightarrow \infty} H = \ell^{-2}$. While the monitonicity of $H$ will hold sufficiently close to the BBBG member of the family by continuity, it fails in general: thus a more careful analysis is required here.   
\end{remark}

The first step in the proof is to parameterise the local asymptotic solutions near $s=0$. The following proposition shows that there is a 2-parameter family of such solutions; we delay the proof until Section \ref{sec:paramend}.

\begin{proposition}\label{prop:ends}
There is a 2-parameter family of smooth solutions to \eqref{eq:BGGGodes} on $(T,\infty)$ for some large $T\in\R$, parameterised by $G_{\infty}:=\lim_{t\to\infty} g^+(t)\in[\ell^{-1},\infty)$ and by $\lambda\in\R$ such that
\begin{align}\label{eq:g+end}
    f^+ & = e^{(\ell^{-1}-G_{\infty})t}\left[\lambda t^{-\frac{5}{2}}+O(t^{-\frac{7}{2}})\right].
\end{align}
The family depends continuously on compact intervals on the parameters $G_{\infty}$ and $\lambda$.
\end{proposition}

Additionally, in what follows, we will use the Theorem of Continuous Dependence on Initial Conditions on a closed interval, so we state it now.

\begin{theorem}[Continuous Dependence on Initial Conditions, Theorem 12.VI \cite{W13}]
\label{thm:ic}
Let $J$ be a compact interval with $t_0\in J$ and let $y_0(t)$ be a solution of the initial value problem
\[y'=f(t,y)\hspace{.5cm} \text{ in } J, \hspace{.5cm} y(t_0)=\eta.\]
Suppose $f(t,y)$ is Lipschitz continuous with respect to $y$ on some open set containing $(t,y_0)$ for $t\in J$.
Then the solution $y_0$ depends continuously on the initial conditions. More precisely, for every $\epsilon>0$, there is $\delta>0$ such that if 
$\vert \zeta-\eta\vert<\delta$, then every solution $z$ of the initial value problem
\[z'=f(t,z),\hspace{.5cm} z(t_0)=\zeta\]
exists in all of $J$ and satisfies
\[\vert z(t)-y_0(t)\vert<\epsilon\hspace{.5cm} \text{ in } J.\]
\end{theorem}

\subsection{Perturbing abelian solutions}\label{sec:peturbation} 

We now prove for any complete, bounded solution of \eqref{eq:BGGGodes} that has $G_{\infty}>\ell^{-1}$, is contained in an open neighbourhood of complete solutions in the parameter space. Here, the parameter space refers to either the initial conditions $(f_1^+,g_1^+)$, or the end conditions  $(G_{\infty}, \lambda)$:  by the Invariance of Domain Theorem, the mapping taking initial conditions to end conditions along complete solutions is a homeomorphism onto its image. 

This shows both that the set of initial conditions corresponding to complete solutions with $G_{\infty}>\ell^{-1}$ is open, and that we can perturb any of the one-parameter family of abelian solutions with  $G_{\infty}>\ell^{-1}$ to get a one-parameter family of non-abelian solutions with $G_\infty$ fixed. 

\begin{proposition} \label{prop:open}
Suppose that a solution of \eqref{eq:BGGGodes} lies in the family of Proposition \ref{prop:ends}, for some $(G_{\infty}, \lambda)$, $G_{\infty}>\ell^{-1}$, and can be extended backwards until it closes smoothly over the singular orbit at $t=0$, corresponding to some initial condition  $(f_1^+,g_1^+)$. Then there are homeomorphic open neighbourhoods of $(f_1^+,g_1^+)$, $(G_{\infty}, \lambda)$, such that solutions in the family with end conditions close to $(G_{\infty}, \lambda)$, can be extended backward until they close smoothly over the singular orbit at $t=0$, and correspond to initial conditions close to $(f_1^+,g_1^+)$.
\end{proposition} 
\begin{proof} To prove this, we will re-write the system \eqref{eq:BGGGodes} for $F^+:= A_1 f^+$, $G^+ := A_3 g^+$: 
\begin{align} \label{eq:BBBGodesinitial1}
 \dot{F}^+ = \frac{F^+}{A_3}\left( 1- \frac{A_1A_3}{B_1 B_3} - G^+ \right)& &\dot{G}^+ = \frac{A_3}{A_1^2}\left( \left( 1- \frac{A_1^2}{B_1^2} \right) G^+ - (F^+)^2 \right)
\end{align}
where, using the formulae in \S \ref{section:G2define}, we have
\begin{align*}
  1- \frac{A_1A_3}{B_1 B_3} = \frac{2(a+p)}{2a + b +p},& & 1- \frac{A_1^2}{B_1^2}=\frac{2 (b+p)}{2a+b+p},& &\frac{A_3}{A_1^2} = \frac{\dot{b}}{2a-b-p}.
\end{align*}
By re-parameterising the local solutions to \eqref{eq:BGGGodes} of Proposition \ref{prop:9}, solutions to \eqref{eq:BBBGodesinitial1} extending smoothly over the singular orbit at $t=0$ are in a continuous two-parameter family $F^+ = \tfrac{1}{2} f^+_1 t^2 + O(t^4)$, $G^+ = \tfrac{1}{2} g^+_1 t^2 + O(t^4)$ in some open neighborhood $[0, t_0)$ of $t=0$. We also have a continuous two-parameter family of solutions on $\left[T, \infty \right)$ by Proposition \ref{prop:ends}. 

We first show that any solution to \eqref{eq:BBBGodesinitial1} on $(0,t_0)$ sufficiently close to the critical point $F^+ = G^+ = 0$ can actually be extended smoothly over $t=0$:  

\begin{lemma} \label{lem:forwardextend} For all $0<t<t_0$, there exists $\delta >0$ such that the map $\mathrm{ev}_t: \R^2 \rightarrow B_\delta(0)$ from the space of initial conditions $(f^+_1,g^+_1)$ to the open ball $B_\delta (0)$ of radius $\delta$ centred at the origin in $\R^2$ is surjective, where $\mathrm{ev}_t$ sends $(f^+_1,g^+_1)$ to the value $\left( F^+ (t), G^+ (t)\right)$ of the corresponding solution to \eqref{eq:BBBGodesinitial1} at time $t$. 
\end{lemma}
\begin{proof} Since this map is continuous and maps initial conditionn $f^+_1 = g^+_1 = 0$ to the critical point $(0,0)$, there must be an open ball around this point contained in the image of $\mathrm{ev}_t$.  
\end{proof}
Note that this map is also injective, by uniqueness of solutions to the singular initial-value problem \eqref{eq:BBBGodesinitial1} running forward in time.

For the rest of the proof, we will show that solutions to \eqref{eq:BBBGodesinitial1} with end conditions sufficiently close to $(G_{\infty}, \lambda)$, a-priori defined only on $[T,\infty)$, can be extended backward for all $t>0$, and converge to critical point $(0,0)$ as $t \rightarrow 0$, and thus they can be smoothly extended over $t=0$ by Lemma \ref{lem:forwardextend}. For backward existence for all $t>0$, we have the following:
\begin{lemma} \label{lemma:backwardextend}
 Fix $T_1>0$. For every $\Gamma>0$ such that $\Gamma^2 < \left. \tfrac{2}{3}\left( 1- \tfrac{A_1^2}{B_1^2} \right)\right|_{t=T_1}$, the set $\lbrace (F^+, G^+) \mid  0<G^+<\tfrac{2}{3}$, $0\leq F^+< \Gamma \rbrace$ is backward-invariant under \eqref{eq:BBBGodesinitial1} for all $0<t<T_1$.   
\end{lemma}
\begin{proof}
 Since $F^+=0$ is an invariant set, and $F^+ = G^+ = 0$ is a critical point, we will focus on the case, $F^+>0$. We have the inequalities  $\tfrac{2}{3}<\frac{2(a+p)}{2a + b +p}< 1, 0<\frac{2 (b+p)}{2a+b+p}<1$, and that  the latter is strictly decreasing.  It follows that $F^+$ is increasing on this set. Since $\dot{G}^+<0$ at $G^+=0$, and $\dot{G}^+>0$ at $G^+=\tfrac{2}{3}$, we have the result as stated. 
\end{proof}

By continuity, this implies that there is an open neighbourhood of $(G_{\infty}, \lambda)$ in the space of end parameters, such corresponding solutions exist for all backward time up to $t=0$. Moreover, this neighbourhood depends only on the solution $(G_{\infty}, \lambda)$. We now show that, up to taking a smaller neighbourhood, these solutions also converge to $F^+=G^+=0$, and hence we can apply Lemma \ref{lem:forwardextend} to extend these solutions smoothly over $t=0$.

To show convergence, we re-parameterise the system \eqref{eq:BBBGodesinitial1} by $s = - \ln t$. This gives the autonomous ODE system: 
\begin{align} \label{eq:BBBGodesinitial2}
\dot{F}^+ =  - 2\left( 1 - G^+ \right) F^+& &\dot{G}^+ = - 2\left( G^+ - (F^+)^2 \right)
\end{align}
up to terms of order $\exp (-s)$ as $s \rightarrow \infty$, uniformly in $F^+, G^+$ contained in the set of Lemma \ref{lemma:backwardextend}. Since the linearisation of the autonomous system \eqref{eq:BBBGodesinitial2} at critical point $\left(0,0\right)$ has repeated eigenvalue $-2$, this critical point is also asymptotically stable (backward in time) for the full system  \eqref{eq:BBBGodesinitial1} inside the invariant set, by \cite[Thm. 2]{markusdiffsystem}. 
\end{proof}

\subsection{Existence of a boundary of solutions} \label{sec:boundary}
We prove the third part of Theorem \ref{prop:bndry} by showing that the limit of every sequence of complete bounded solutions is also a complete bounded solution, and hence the set of complete bounded solutions is closed.

Suppose that $(f^+,g^+)_n$ is a sequence of complete bounded solutions to \eqref{eq:BGGGodes} with initial conditions $(f_1^+,g_1^+)_n$ satisfying $(g_1^+)_n > (g_1^+)_{n+1}$ and $(f_1^+)_n < (f_1^+)_{n+1}$. Denote by $(G_{\infty})_n$ the sequence whose $n^{th}$ term is $\lim_{t\to\infty}(g^+)_n$ and let $(\lambda)_n$ be the value of $\lambda$ for each $(f^+)_n$, written in the form \eqref{eq:g+end}. Then Lemma \ref{lem:orderoflims} implies that $(G_{\infty})_n>(G_{\infty})_{n+1}$ so $(G_{\infty})_n$ is decreasing and hence has a finite limit greater than or equal to $\ell^{-1}$. $(\lambda)_n$ is a sequence of positive real numbers, and we see from \eqref{eq:BGGGodes} that it must be bounded since $g^+$ and $\dot{g}^+$ are both bounded. Hence, we can pass to a convergent subsequence, which we will also call $(G_{\infty},\lambda)_n$, and only consider this sequence for the remainder of the argument.

We now compare two local solutions. The first is defined on $(T_1,\infty)$ and is the solution of Proposition \ref{prop:ends} given by $(G_{\infty}^1,\lambda^1)$ where \[G_{\infty}^1=\lim_{n\to\infty}(G_{\infty})_n = \lim_{n\to\infty}\lim_{t\to\infty}(g^+)_n \hspace{0.5cm}\text{ and }\hspace{.5cm} \lambda^1=\lim_{n\to\infty}\lim_{t\to\infty}\left( e^{((G_{\infty})_n-1)t}t^{\frac{5}{2}}f^+\right)\]
while the second is the solution $(f_{\infty}^+,g_{\infty}^+)$ of Proposition \ref{prop:9}, defined on $[0,T_2)$ and given by \[(f_1^+)_{\infty}=\lim_{n\to\infty}(f_1^+)_n \hspace{.5cm} \text{ and } \hspace{.5cm} (g_1^+)_{\infty}=\lim_{n\to\infty}(g_1^+)_n.\]

\begin{proposition}
$(f^+_{\infty}, g^+_{\infty})$ extends to a solution on $[0,\infty)$ with $\lim_{t\to\infty}g_{\infty}^+=G^1_{\infty}$.  
\end{proposition}
\begin{proof}
We will show that, by the Theorem \ref{thm:ic}, one can extend the interval of existence of $(f^+_{\infty}, g^+_{\infty})$ to some arbitrarily large time $T \in (T_1,\infty)$ by taking $n$ sufficiently large, and then take a limit of the sequence at fixed $T$. Explicitly, let $z_1$ be the $g^+$-component of the solution given by the end conditions $(G_{\infty}^1,\lambda^1)$ and $z_2=g_{\infty}^+$. Choose some large time $T$ that lies in the interval of existence of $z_1$.

Since the sequence $(g_1^+)_n$ of initial conditions converges, then for every $\delta>0$, there exists $N$ such that for all $n>N$, 
\[\vert (g_1^+)_n - (g_1^+)_{\infty}\vert<\delta.\]
Together with Theorem \ref{thm:ic}, for every $\epsilon>0$, there exists $\delta_1$ and hence $N_1$ such that for all $n>N_1$,
\[\vert (g_1^+)_n - (g_1^+)_{\infty}\vert<\delta_1 \implies\vert (g^+)_n(T)-(g^+)_{\infty}(T)\vert=\vert (g^+)_n(T)-z_2(T)\vert<\frac{\epsilon}{2}.\]
Since the sequence $(G_{\infty})_n$ converges, then for every $\delta>0$, there exists $N$ such that for all $n>N$, 
\[\vert (G_{\infty})_n - G_{\infty}^1\vert<\delta.\]
By continuous dependence on $G_{\infty}$ from Proposition \ref{prop:ends}, for every $\epsilon>0$, there exists $\delta_2$ and hence $N_2$ such that for all $n>N_2$,
\[\vert (G_{\infty})_n - G_{\infty}^1\vert<\delta_2\implies\vert (g^+)_n(T)-z_1(T)\vert<\frac{\epsilon}{2}.\]
Hence, for all $\epsilon>0$ and all $n>\max\{N_1,N_2\}$,
\[\vert z_1(T)-z_2(T)\vert \leq \vert z_1(T)-(g^+)_n(T)\vert + \vert (g^+)_n(T)-z_2(T)\vert < \epsilon.\]
So $z_1(T)=z_2(T)$ and hence the solutions $z_1$ and $z_2$ coincide on $[T,\infty)$. Therefore $g_{\infty}^+$ is defined on $[0,\infty)$ and $\lim_{t\to\infty}g_{\infty}^+=G_{\infty}^1$.
\end{proof}

%If $f^+_{\infty}$ is unbounded, then because it is bounded above by Lemma \ref{lem:botright}, we see that there must exist a $T>0$ such that $f^+_{\infty}(T)+\epsilon<0$ for some $\epsilon>0$. Then, on a compact interval of time containing $T$, Theorem \ref{thm:ic} tells us that there is some $N\in\mathbb{N}$ such that for all $n>N$, 
%\[\vert (f^+)_n(T)-f^+_{\infty}(T)\vert<\epsilon.\]
%But since $(f^+)_n(T)>0$, we have $(f^+)_n(T)-f^+_{\infty}(T)>\epsilon$, yielding a contradiction.
 We have shown that if we have a sequence of complete bounded solutions whose initial conditions $(f_1^+,g_1^+)_n$ satisfy $(g_1^+)_n > (g_1^+)_{n+1}$ and $(f_1^+)_n < (f_1^+)_{n+1}$, then the limit of this sequence is also a complete bounded solution. By Lemma \ref{lem:orderoflims}, every initial condition leading to a complete, bounded solution must arise as the limit of a such a sequence, so this implies that the region of initial conditions which lead to complete bounded solutions is closed.   

By Proposition \ref{prop:open}, it follows that the boundary must correspond to solutions with $G_{\infty} = \ell^{-1}$. Finally, we claim that this boundary is a graph of the function $g_1^+\mapsto f(g_1^+)$. Suppose that there are two points on the boundary with the same value of $g_1^+$; then one has a smaller value of $f_1^+$ and hence by Lemma \ref{lem:orderoflims}, the solution has a larger value of $G_{\infty}$. This contradicts these initial conditions being on the boundary of those which lead to complete bounded solutions. So the boundary can be described as a graph of a continuous function $g_1^+\mapsto f_1^+(g_1^+)$. Hence, subject to proving Proposition \ref{prop:ends}, we have proved Theorem \ref{prop:bndry}.

\begin{comment}
\begin{remark}
Recall that these $G_2$-instantons correspond to a CY monopole with mass
\[m=\lim_{t\to\infty}\left\vert\frac{A_3}{2}g^+T_1\right\vert=\frac{G_{\infty}}{2}.\]
Hence we can describe the region of initial conditions which lead to complete bounded solutions as the union of a line of abelian solutions and a set with boundary given by the initial conditions of $G_2$-instantons asymptotic to the connection $A_{\infty}$, defined in Section \ref{sec:mainres}, with $g_1^+>0$ and corresponding CY monopoles with mass $m=\frac{1}{2}$, as shown in Figure \ref{fig:boundaryintro}. 

The set of initial conditions in this union corresponding to abelian instantons give solutions which are rigid for each fixed mass $m$. On the other hand, for each $m\geq\frac{1}{2}$ and $g_1^+>0$, we are able to perturb the abelian solutions in a 1-parameter family of $G_2$-instantons with corresponding CY monopoles of the same mass. When $m=\frac{1}{2}$, the perturbations have $f^+$ decaying asympotically at a polynomial rate, while for $m>\frac{1}{2}$, it decays exponentially fast.\hfill$\Diamond$
\end{remark}
\end{comment}

\subsection{Parameterising end solutions}\label{sec:paramend}

In this section, we prove Proposition \ref{prop:ends}, a parameterisation of the end solutions, i.e. local solutions of \eqref{eq:BGGGodes} on $(T,\infty)$ for some $T\in\R$. 

% \begin{itemize}
%     \item We start with the case of $G_{\infty}=\ell^{-1}$, since it is a simpler situation and we can apply the usual singular initial value problem arguments to yield a 1-parameter family of local solutions, c.f. \cite[Theorem 4.3]{FHN18}. 
    
%     \item We then consider the case when $G_{\infty}$ is any real number and see that the ODEs take the form of an irregular singular IVP (ISIVP). For each $G_{\infty}$, there is an abelian solution with $f^+=0$ given in Proposition \ref{lem:abBGGG}.
    
%     \item For a fixed $G_{\infty}>\ell^{-1}$, we find in Proposition \ref{lem:smoothz} a 1-parameter family of solutions which differ from the abelian solution by an exponentially small function and show that this 2-parameter family gives the only local solutions to the ISIVP. This family is characterised by the properties that $g^+\to G_{\infty}$ and $f^+$ decays exponentially with rate $\ell^{-1}-G_{\infty}$ as $t\to\infty$; the family depends continuously on both parameters.
% \end{itemize}
\begin{proof}
Recall that, for general $G_{\infty} \in \R$, there is an abelian solution to \eqref{eq:BGGGodes} given in Proposition \ref{lem:abBGGG}, which can be deformed to a non-abelian solution only if $G_{\infty}\geq \ell$ by Lemma \ref{lem:bounded}. We set up the ODE for $G_{\infty}\geq \ell$ as a singular initial-value problem using the formulation of Section \ref{sec:existresBGGG} and, in particular, equations \eqref{eq:f+}, \eqref{eq:g+}, and \eqref{eq:BGGGasym}. Using Lemma \ref{lem:expdecay}, we can write
\[f^+(t)=t^{-\frac{5}{2}} \exp (t (\ell-G_\infty) ) X(t) \ \ \ \ \text{ and }\ \ \ \ g^+(t)=G_{\infty}-t^{-2}(\tfrac{9}{4}\ell^2G_{\infty}+Y(t))\]
for some bounded functions $(X,Y)$. Then for $s=\frac{1}{t}$, we can rewrite \eqref{eq:BGGGodes} as the $G_{\infty}$-dependent system
\begin{equation} \label{eq:ISIVP}
\begin{aligned}
    \frac{dX}{ds}&=(\gamma_2-\tfrac{9}{4}\ell^2G_{\infty}-Y)X,\\
    \frac{dY}{ds}&=-2s^{-1}Y-G_{\infty}\gamma_1-(\tfrac{81}{8}\ell^4G_{\infty}+\tfrac{9}{2}\ell^2Y+X^2 \exp(\tfrac{2 (\ell-G_\infty)}{s}))s+\gamma_1(Y+\tfrac{9}{4}\ell^2G_{\infty})s^2.
\end{aligned} 
\end{equation}
As $s \rightarrow 0$ we have,
\begin{align}
\frac{dX}{ds}= O(1),& & \frac{dY}{ds} = - 2 s^{-1} Y + O(1).    
\end{align}  
Thus we can find local solutions to \eqref{eq:ISIVP} by applying the following theorem.

\begin{theorem}[{\cite[Theorem 4.3]{FHN18}}]
\label{thm:4.3}
Consider the singular initial value problem
\[\dot{y}=\frac{1}{t}M_{-1}(y)+M(t,y),\hspace{1cm} y(0)=y_0,\]
where $y$ takes values in $\R^k$, $M_{-1}:\R^k\to\R^k$ is a smooth function of $y$ in a neighbourhood of $y_0$ and $M:\R\times\R^k\to\R^k$ is smooth in $t,y$ in a neighbourhood of $(0,y_0)$. Assume that
\begin{enumerate}
    \item[\textup{(i)}] $M_{-1}(y_0)=0$;
    \item[\textup{(ii)}] $hId-d_{y_0}M_{-1}$ is invertible $\forall\,h\in\mathbb{N}$, $h\geq 1$.
\end{enumerate}
Then there exists a unique solution $y(t)$ in a sufficiently small neighbourhood of 0. Furthermore, $y$ depends continuously on $y_0$ satisfying $(i)$ and $(ii)$.
\end{theorem}

In the case of \eqref{eq:ISIVP}, once we fix initial condition $X(0)= \lambda \in \R$, $Y(0)=0$, then we have $d_{y(0)}M_{-1}=\text{diag}(0,-2)$. Hence, for all $h\in\mathbb{N}_{\geq1}$, $h\text{Id}-d_{y(0)}M_{-1}$ is invertible, and by Theorem \ref{thm:4.3}, there exists a unique solution in a neighborhood $\left[0, \epsilon\right)$ of $s=0$. Moreover, it depends continuously on the parameters $(G_\infty, \lambda)$, and thus we have proved Proposition \ref{prop:ends}

\end{proof}

\section{Adiabatic limits: Instantons on Taub-Nut}
\label{sec:taubnut}
One can view the ALC $\mathbb{B}_7$-family as a higher-dimensional analogue of the hyperk\"{a}hler Taub-NUT metric on $\R^4$. The Taub-NUT metric is also of co-homogeneity one, and is asymptotic to a circle-fibration of fixed length over flat $\R^3$. Moreover, as we will show in this section, in a suitable adiabatic limit of the $\mathbb{B}_7$-family, the $G_2$-instantons constructed in the previous section appear as anti-self dual instantons on Taub-NUT fibred over the associative $S^3$.

First, we will briefly discuss $SU(2)\times U(1)$-invariant Yang-Mills instantons with gauge group $SU(2)$ on the Taub-NUT space. These instantons were mostly constructed in a series of papers by Chakrabarti et al. \cite{bccc4:1980, BJCC:YMexplicit5}, c.f. \cite{POPE1978424}, \cite{etesi:2001}, \cite{Kim2000169}, and can be given explicitly using ADHM data \cite{cherkis:2010}.  Here, we give some simple explicit formulae, by finding the general solution of the $SU(2) \times U(1)$-invariant\footnote{In fact, one can show a-posteriori that complete $SU(2)$-invariant solutions must be additionally $U(1)$-invariant, although we will not need this here.} ansatz considered in \cite{Kim2000169}.  

We recall some basic facts about the Taub-NUT space, following \cite{atiyah:monopoles}. The Taub-NUT space admits an isometric action by $SU(2)$ of co-homogeneity one, with isotropy subgroups $\lbrace 1 \rbrace \subset SU(2) \subseteq SU(2)$. We can write the metric, parameterised by radial geodesic arc-length $t$, as follows: 
\begin{equation}
 g =  dt^{2} + f_1^2 \left[\left(e^1\right)^2 + \left(e^2\right)^2 \right] +  f_3^2 \left(e^3\right)^2
 \end{equation}
where $e^1$, $e^2$, $e^3$ is a basis left-invariant $1$-forms on $SU(2)$ satisfying $d e^i = - e^j \wedge e^k$, for $(i j k)$ cyclic permutations of $(123)$, and $f_1(t)$, $f_3(t)$ are two functions satisfying the ODE:
\begin{align} \label{eq:ODE1}
\dot{f}_1 = \frac{1}{2}\left(2 - \frac{f_3}{f_1}\right)& &\dot{f}_3 =  \frac{1}{2}\left(\frac{f_3}{f_1}\right)^2. 
\end{align}
Complete solutions to \eqref{eq:ODE1} are parameterised by a positive constant $m$. By introducing the variable $\eta:\left[0,\infty\right) \rightarrow \left(\infty, m^{-2} \right)$, such that $dt = -\eta^{\tfrac{1}{2}}\left(\eta- \tfrac{1}{m^2}\right)^{-2} d \eta $, they are given explicitly by:
\begin{align} \label{eq:taubNUT}
f_1 =  \eta^{\tfrac{1}{2}}\left(\eta - \tfrac{1}{m^2}\right)^{-1}& &f_3 = \eta^{-\tfrac{1}{2}}.
\end{align}
The resulting model metric at infinity is a circle of radius $m$ fibred over $\R^3$. In particular, as $t\rightarrow \infty$, we have:
\begin{align} 
f_1 =  t + O(t^{-1})& &f_3 = m + O(t^{-1}).
\end{align} 
\begin{remark} As we will need it later, we note we can expand \eqref{eq:taubNUT} in $t$ near the the origin at $t=0$ as:
\begin{align}  \label{eq:taubNUTexp}
f_1 = \tfrac{1}{2} t + \tfrac{1}{24 m^2} t^3 + O(t^5)& &f_3 = \tfrac{1}{2} t - \tfrac{1}{12 m^2} t^3 + O(t^5).
\end{align}
\end{remark}
Note that the parameter $m$ can be fixed by scaling the resulting metric, which appears as the symmetry 
\begin{align*}
   f_1(t) \mapsto \lambda f_1 \left( \tfrac{t}{\lambda} \right)& &f_3(t) \mapsto \lambda f_3\left( \tfrac{t}{\lambda} \right)  
\end{align*}
of \eqref{eq:ODE1}. 

\subsection{$SU(2) \times U(1)$-invariant instantons on Taub-NUT} \label{sec:taubnutasd}
Anti-self dual (ASD) instantons are special Yang-Mills connections $A$ on a Riemannian four-manifold $(M,g)$, satisfying:
\begin{align}
    F_A = - * F_A.
\end{align}
As is shown in \cite{Kim2000169}, $SU(2) \times U(1)$-invariant anti-self dual instantons with gauge group $SU(2)$ on Taub-NUT space are given in the radial gauge by the invariant connection:
\begin{align} \label{eq:invariantconnection}
A = a_1(t) (E_1 e^1 + E_2 e^2) + a_3(t) E_3 e^3
\end{align}
with $E_i$ denoting the left-invariant vector field on $SU(2)$ which is dual to $e^i$, and $a_1, a_3$  two functions satisfying the ODE: 
\begin{align} \label{eq:asdodes0}
\dot{a}_1 = - \tfrac{1}{f_3} \left( a_3 - 1 \right) a_1 & &\dot{a}_3 = - \tfrac{f_3}{f_1^2} \left( a_1^2 - a_3 \right).
\end{align}
In their beautiful paper \cite{etesi:2001}, Etesi and Hausel find an explicit one-parameter family of solutions to \eqref{eq:asdodes0} using geometric methods, generalising the single solution of \cite{POPE1978424}, rediscovered in \cite{Kim2000169}. In fact, one can write down the two-parameter family of general solutions of \eqref{eq:asdodes0}, by noticing that the system \eqref{eq:asdodes0} has a  conserved quantity: 
\begin{align}
(\eta-m^{-2})^2 \tfrac{d}{d\eta}  (\eta a_3)- (\eta a_3 -m^{-2})^2.
\end{align}
Integrating this out, and solving the resulting separable ODE, gives the full two-parameter family of complete solutions to \eqref{eq:asdodes0}, 
    \begin{align} \label{eq:asdsoln1}
a_1 = \tfrac{C}{\eta - m^{-2}} \mathrm{csch} \left(\tfrac{C}{\eta - m^{-2}} + D \right)& &a_3 = \tfrac{1}{\eta}\left( m^{-2} +  C \coth \left(\tfrac{C}{\eta - m^{-2}} + D\right) \right)
\end{align}
for non-negative constants $C, D$. For each $C\geq0$, these solutions are asymptotic to the one-parameter family of reducible abelian solutions of  \eqref{eq:asdodes0}: 
\begin{align} \label{eq:asdsoln0}
a_1 = 0& &a_3=\tfrac{1}{\eta}\left( m^{-2} +  C \right).   
\end{align}
Note that we are free to take $C<0$ for abelian solutions, but these do not admit non-abelian perturbations as solutions of \eqref{eq:asdodes0}. 

We can recover the solutions of \cite{etesi:2001} from the family \eqref{eq:asdsoln1}, by taking the limit $C \rightarrow 0$, $D \rightarrow 0$ such that $B:=\sinh(D)/C \geq 0$ is held fixed. In this limit, we obtain the one-parameter family of solutions to \eqref{eq:asdodes0}: 
\begin{align} \label{eq:asdsoln2}
a_1 = \tfrac{1}{1+ B (\eta - m^{-2})}& &a_3 = \tfrac{1}{\eta}\left( m^{-2} +   \tfrac{ (\eta -m^{-2})}{1+ B (\eta -m^{-2})}\right)
\end{align} 
asymptotic to the flat Maurer-Cartan form $a_1=a_3=1$. 

Finally, we claim that the solutions \eqref{eq:asdsoln0}, \eqref{eq:asdsoln2}, and \eqref{eq:asdsoln1} for $D>0$, all define smooth connections on the bundle of anti-self-dual forms over Taub-NUT, while \eqref{eq:asdsoln1} gives a one-parameter family of connections on the bundle of self-dual forms when $D=0$. Respectively, these $SU(2)$-equivariant bundles correspond to the trivial and the adjoint action of the isometric $SU(2)$ on its lie algebra, so this claim can be checked by a routine computation, which we will not repeat here, see e.g. \cite[Appendix A]{stein:calabiyaugauge}. 

\begin{remark} For later reference, we note that we can expand \eqref{eq:asdsoln1} with $D>0$ near  $t=0$ as:
\begin{align} \label{eq:asdsols}
a_1 = \mu_1 t^2 + O(t^4) & &a_3 =\mu_3 t^2 + O(t^4) 
\end{align}
where $\mu_1= \tfrac{C}{4 m^2} \mathrm{csch} (D)$, $\mu_3= \tfrac{C}{4m^2} \mathrm{coth} (D) + \tfrac{1}{m^2}$ are subject to the constraint
\begin{align*}
 \left(\mu_3 - \tfrac{1}{4m^2}\right)^2 - \mu_1^2  = \left( \tfrac{C}{4m^2} \right)^2.  
\end{align*} 
\end{remark}

\subsection{Adiabatic limit of $\Bsev$-family.} \label{sec:adiabatic}

We will now show that we can recover the ASD instantons \eqref{eq:asdsoln1} by considering sequences of complete solutions to $G_2$-instanton equations \eqref{eq:BGGGodes} in an adiabatic limit, as the $\Bsev$-family of metrics approaches a rescaled copy of Taub-NUT, fibred along a round $S^3$. 

Firstly, we show that by keeping the freedom to vary the overall scale of the metric, while fixing the size of the singular orbit $S^3$ at $t=0$, we can indeed recover this adiabatic limit uniformly on compact subsets:   

\begin{prop} \label{prop:ALFfibration} Let $\left(A_i, B_i\right)$ be a complete solution to \eqref{eq:hitchFHNmetric} in the one-parameter $\mathbb{B}_7$-family with $m^2 =(4r_0^3 (2\bar{a} - \bar{b}))^{-1}$ fixed. Define  $A_i^\lambda (t) := \lambda ^{-2} A_i ( \lambda^2 t), B_i^\lambda (t) := \lambda ^{-1} B_i ( \lambda^2 t )$. Then if $\lambda = r_0$, we have $B_i^\lambda \rightarrow 2$, $A_i^\lambda \rightarrow f_i$ as $\lambda \rightarrow 0$ uniformly on compact intervals, where $f_i$ are the functions given in \eqref{eq:taubNUT}. 
\end{prop}
\begin{proof}
Recall the Hitchin flow equations \eqref{eq:hitchFHNmetric} give ODEs for the functions $A_1, B_1, A_3, B_3$ appearing as coefficients of the  metric. Now, for any constant $\lambda>0$,   by re-parameterising, solutions of \eqref{eq:hitchFHNmetric} are in one-to-one correspondence with solutions of the system:
\begin{gather} \label{eq:rescaledB7}
\begin{aligned}
    \dot{A}_1^\lambda & = \frac{1}{2}\left(\frac{(B_1^\lambda)^2+(B_3^\lambda)^2-\lambda^2 (A_1^\lambda)^2}{B_1^\lambda B_3^\lambda}-\frac{A_3^\lambda}{A_1^\lambda}\right), \\
    \dot{A}_3^\lambda & = \frac{1}{2}\left(\frac{(A_3^\lambda)^2}{(A_3^\lambda)^2}-\lambda^2 \frac{(A_3^\lambda)^2}{(B_1^\lambda)^2}\right),\\
    \dot{B}_1^\lambda & = \frac{1}{2}\left(\frac{\lambda^2 (A_1^\lambda)^2 +(B_3^\lambda)^2-(B_1^\lambda)^2}{A_1^\lambda B_3^\lambda}+ \lambda^2 \frac{A_3^\lambda}{B_1^\lambda}\right), \\
    \dot{B}_3^\lambda & = \frac{\lambda^2 (A_1^\lambda)^2 +(B_1^\lambda)^2-(B_3^\lambda)^2}{A_1^\lambda B_1^\lambda}.
\end{aligned}
\end{gather}
In particular, if we let $\lambda=r_0>0$, then for fixed $r_0$, this system has a one-parameter family of solutions extending over the singular orbit at $r=0$ by re-parameterising the family of solutions to \eqref{eq:hitchFHNmetric} given by the $\mathbb{B}_7$-family with $r_0$ fixed. We will claim this system admits a one-parameter family of solutions all the way to $\lambda =0$. 

To show this, we use the boundary extension conditions to the singular orbit in \cite[Lemma 8, Appendix A]{LO18b}, which imply that we can write the re-parameterised solutions to \eqref{eq:hitchFHNmetric} as $B^\lambda_i = 2 + t^2 b_i$, $A^\lambda_i = \tfrac{t}{2} + t^3 a_i$ for some smooth functions $a_i$, $b_i$. Using this, \eqref{eq:rescaledB7} gives a singular initial-value problem in $a_i, b_i$, given up to $O(t^2)$ terms, by:
\begin{align*}
t\dot{a}_1 &= -\tfrac{\lambda^2}{32} - a_3 - 2a_1& &t\dot{a}_3 = -\tfrac{\lambda^2}{32} - a_3 - 2a_1 \\
t\dot{b}_1 &= \tfrac{\lambda^2}{4} +2b_3 - 4b_1& &t\dot{b}_3=\tfrac{\lambda^2}{4} +4b_1 - 6b_3
\end{align*}
where $b_1(0) = b_3(0) = -4 \left( a_3 (0) + 2 a_1(0)\right) = \tfrac{\lambda^2}{8}$. This system has the linearisation given by the matrix 
\begin{align*}
    \begin{pmatrix}
-2 & -1 & &\\
-2 &-1 & & \\
 & &-4 & 2\\
 & & 4 & 6 
\end{pmatrix}
\end{align*}
with eigenvalues $0,-3,-2,-8$. The rescaled solutions of \eqref{eq:hitchFHNmetric} satisfy $ a_3 (0) = - 2 r_0^3 (2\bar{a} - \bar{b})$, thus, we get a one-parameter family by fixing this constant as we take $\lambda = r_0 \rightarrow 0$. 

In this limit, the re-scaled equations \eqref{eq:rescaledB7} for $A_i$ decouple to give the ODEs \eqref{eq:ODE1} for the metric coefficients on Taub-NUT. Thus, using the expansion  \eqref{eq:taubNUTexp}, we can identify this one-parameter family of solutions to \eqref{eq:rescaledB7} with the one-parameter family \eqref{eq:taubNUT} by taking $B_1 =B_3 = 2$, and $m^2 =(4r_0^3 (2\bar{a} - \bar{b}))^{-1}$. 
\end{proof}

With this adiabatic limit in hand, we now show that we can recover \eqref{eq:asdsoln1} from the $G_2$-instanton ODE, written in the form \eqref{eq:BBBGodesinitial1}. Recall that these have a two-parameter family of solutions $F^+ = \tfrac{1}{2} f^+_1 t^2 + O(t^4)$, $G^+ = \tfrac{1}{2} g^+_1 t^2 + O(t^4)$ near $t=0$:
\begin{theorem} \label{theorem:asdconvergance} Let $\left(F^+,G^+\right)$ be the complete solution to \eqref{eq:BBBGodesinitial1} with $\mu_1 :=\tfrac{r_0^4}{2}f_1^+$,   $\mu_3:= \tfrac{r_0^4}{2}g_1^+ $ fixed, subject to  $\mu_3 - \tfrac{1}{4m^2}\geq \mu_1\geq 0 $, for the one-parameter $\mathbb{B}_7$-family with $m^2 =(4r_0^3 (2\bar{a} - \bar{b}))^{-1}$ fixed. For any $\lambda>0$, let $F_+^\lambda (t) := F^+ ( \lambda^2 t)$, $G_+^\lambda (t) := G^+ ( \lambda^2 t )$. If $\lambda = r_0$, then as $\lambda \rightarrow 0$, then  $\left(F_+^\lambda, G_+^\lambda \right)$ converges to the solution \eqref{eq:asdsoln1} given by the expansion  \eqref{eq:asdsols} near $t=0$. 
\end{theorem} 
\begin{proof}
For each $\lambda>0$, rescaling \eqref{eq:BBBGodesinitial1} gives the system of ODEs:
    \begin{align*} 
 \dot{F}_+^\lambda &= \frac{F_+^\lambda}{A_3^\lambda}\left( 1- \lambda^2 \frac{A_1^\lambda A_3^\lambda}{B_1^\lambda B_3^\lambda} - G_+^\lambda \right) \\ \nonumber
 \dot{G}_+^\lambda &= \frac{A_3^\lambda}{(A_1^\lambda)^2}\left( \left( 1- \lambda^2 \frac{(A_1^\lambda)^2}{(B_1^\lambda)^2} \right) G_+^\lambda  - (F_+^\lambda )^2 \right).
\end{align*}
By our previous result in Proposition \ref{prop:ALFfibration}, this system converges smoothly to the system for anti-self dual connections on Taub-NUT \eqref{eq:asdodes0} in the limit $\lambda \rightarrow 0$.

To show the re-scaled solutions also converge smoothly in this limit, we use the extension conditions to the singular orbit in \cite[Appendix A.2, Lemma 10]{LO18b}, to define smooth functions $\tilde{F}, \tilde{G}$ such that $F_+^\lambda= t^2 \tilde{F}$, $G_+^\lambda= t^2 \tilde{G}$. These solve a smooth initial value-problem $\tilde{F}= O(t), \tilde{G} = O(t)$, which has a two-parameter family of solutions by fixing initial condition $\tilde{F}(0), \tilde{G}(0)$, for any $\lambda\geq 0$. The re-parameterised solutions to \eqref{eq:BBBGodesinitial1} have $\tilde{F}(0) = \tfrac{1}{2} r_0^4 f_1^+$,  $\tilde{G}(0) = \tfrac{1}{2}r_0^4g_1^+$, so keeping these constants fixed while taking $r_0 \rightarrow 0$ gives the two-parameter family of solutions \eqref{eq:asdsoln1}. 
\end{proof}

\bibliographystyle{alpha}
\bibliography{references}

\begin{thebibliography}{BJCC80b}

\bibitem[AH88]{atiyah:monopoles}
Michael Atiyah and Nigel Hitchin.
\newblock {\em The geometry and dynamics of magnetic monopoles}.
\newblock M. B. Porter Lectures. Princeton University Press, Princeton, NJ, 1988.

\bibitem[BGGG01]{bggg}
Andreas Brandhuber, Jaume Gomis, Steven~S. Gubser, and Sergei Gukov.
\newblock Gauge theory at large {$N$} and new {$G_2$} holonomy metrics.
\newblock {\em Nuclear Phys. B}, 611(1-3):179--204, 2001.

\bibitem[BJCC80a]{bccc4:1980}
H.~Boutaleb-Joutei, A.~Chakrabarti, and A.~Comtet.
\newblock Gauge field configurations in curved space-times. {IV}. {S}elf-dual {${\rm SU}(2)$}\ fields in multicenter spaces.
\newblock {\em Phys. Rev. D (3)}, 21(8):2280--2284, 1980.

\bibitem[BJCC80b]{BJCC:YMexplicit5}
H.~Boutaleb-Joutei, A.~Chakrabarti, and A.~Comtet.
\newblock Gauge field configurations in curved space-times. {V}. {R}egularity constraints and quantized actions.
\newblock {\em Phys. Rev. D (3)}, 21(8):2285--2290, 1980.

\bibitem[Bog13]{bogoyavg2}
O.~A. Bogoyavlenskaya.
\newblock On a new family of complete {R}iemannian metrics on {$S^3\times\Bbb R^4$} with holonomy group {$G_2$}.
\newblock {\em Sibirsk. Mat. Zh.}, 54(3):551--562, 2013.

\bibitem[CGLP02]{gibbonspope:g2mtheory}
M.~Cveti\v{c}, G.~W. Gibbons, H.~L\"{u}, and C.~N. Pope.
\newblock A {$G_2$} unification of the deformed and resolved conifolds.
\newblock {\em Phys. Lett. B}, 534(1-4):172--180, 2002.

\bibitem[CH21]{cherkisbows}
Sergey~A. Cherkis and Jacques Hurtubise.
\newblock Instantons and bows for the classical groups.
\newblock {\em Q. J. Math.}, 72(1-2):339--386, 2021.

\bibitem[Che10]{cherkis:2010}
Sergey~A. Cherkis.
\newblock Instantons on the {T}aub-{NUT} space.
\newblock {\em Adv. Theor. Math. Phys.}, 14(2):609--641, 2010.

\bibitem[DS11]{donaldson2009gauge}
Simon Donaldson and Ed~Segal.
\newblock Gauge theory in higher dimensions, {II}.
\newblock In {\em Surveys in differential geometry. {V}olume {XVI}. {G}eometry of special holonomy and related topics}, volume~16 of {\em Surv. Differ. Geom.}, pages 1--41. Int. Press, Somerville, MA, 2011.

\bibitem[DT98]{donaldson1996gauge}
Simon Donaldson and Richard Thomas.
\newblock Gauge theory in higher dimensions.
\newblock In {\em The geometric universe ({O}xford, 1996)}, pages 31--47. Oxford Univ. Press, Oxford, 1998.

\bibitem[EH01]{etesi:2001}
G\'abor Etesi and Tam\'as Hausel.
\newblock Geometric construction of new {Y}ang-{M}ills instantons over {T}aub-{NUT} space.
\newblock {\em Phys. Lett. B}, 514(1-2):189--199, 2001.

\bibitem[FHN21a]{FoscoloALC}
Lorenzo Foscolo, Mark Haskins, and Johannes Nordstr\"{o}m.
\newblock Complete noncompact {$G_2$}-manifolds from asymptotically conical {C}alabi-{Y}au 3-folds.
\newblock {\em Duke Math. J.}, 170(15):3323--3416, 2021.

\bibitem[FHN21b]{FHN18}
Lorenzo {Foscolo}, Mark {Haskins}, and Johannes {Nordstr{\"o}m}.
\newblock {Infinitely many new families of complete cohomogeneity one $G_2$-manifolds: $G_2$ analogues of the Taub-NUT and Eguchi-Hanson spaces}.
\newblock {\em J. Eur. Math. Soc.}, 23(7):2153--2220, 2021.

\bibitem[KL20]{lotay:g2conifolds}
Spiro Karigiannis and Jason~D. Lotay.
\newblock Deformation theory of {$\rm G_2$} conifolds.
\newblock {\em Comm. Anal. Geom.}, 28(5):1057--1210, 2020.

\bibitem[KY00]{Kim2000169}
Hongsu Kim and Yongsung Yoon.
\newblock Yang-mills instantons in the gravitational instanton backgrounds.
\newblock {\em Physics Letters, Section B: Nuclear, Elementary Particle and High-Energy Physics}, 495(1-2):169 – 175, 2000.

\bibitem[LO18]{LO18b}
Jason~D. Lotay and Goncalo Oliveira.
\newblock {$\SU(2)^2$-invariant $G_2$-instantons}.
\newblock {\em Mathematische Annalen}, 371(1):961--1011, Jun 2018.

\bibitem[Mar56]{markusdiffsystem}
L.~Markus.
\newblock Asymptotically autonomous differential systems.
\newblock In {\em Contributions to the theory of nonlinear oscillations, vol. 3}, Annals of Mathematics Studies, no. 36, pages 17--29. Princeton University Press, Princeton, N.J., 1956.

\bibitem[MNT22]{MNT22}
Karsten Matthies, Johannes Nordstr\"{o}m, and Matt Turner.
\newblock {$SU(2)^2\times U(1)$}-invariant {$G_2$}-instantons on the {AC} limit of the {$\mathbb{C}_7$} family, 2022.
\newblock arXiv:2202.05028.

\bibitem[MS12]{MS12}
Thomas~Bruun Madsen and Simon~M. Salamon.
\newblock Half-flat structures on ${S}^3\times {S}^3$.
\newblock {\em Annals of Global Analysis and Geometry}, 44:369--390, 2012.

\bibitem[PY78]{POPE1978424}
C.N. Pope and A.L. Yuille.
\newblock A yang-mills instanton in taub-nut space.
\newblock {\em Physics Letters B}, 78(4):424--426, 1978.

\bibitem[Sal89]{salamon:riemanniangeometry}
Simon Salamon.
\newblock {\em Riemannian geometry and holonomy groups}, volume 201 of {\em Pitman Research Notes in Mathematics Series}.
\newblock Longman Scientific \& Technical, Harlow; copublished in the United States with John Wiley \& Sons, Inc., New York, 1989.

\bibitem[SEW15]{walpuski:tcsinstantons}
Henrique~N. S\'{a}~Earp and Thomas Walpuski.
\newblock {$\rm {G}_2$}-instantons over twisted connected sums.
\newblock {\em Geom. Topol.}, 19(3):1263--1285, 2015.

\bibitem[ST24]{ST23}
Jakob Stein and Matt Turner.
\newblock {$G_2$}-instantons on the spinor bundle of the 3-sphere.
\newblock {\em J. Geom. Anal.}, 34(5):Paper No. 149, 22, 2024.

\bibitem[Ste23]{stein:calabiyaugauge}
Jakob Stein.
\newblock {$SU(2)^2$}-{I}nvariant {G}auge {T}heory on {A}symptotically {C}onical {C}alabi-{Y}au 3-{F}olds.
\newblock {\em J. Geom. Anal.}, 33(4):121, 2023.

\bibitem[Tia00]{tian2000gauge}
Gang Tian.
\newblock Gauge theory and calibrated geometry. {I}.
\newblock {\em Ann. of Math. (2)}, 151(1):193--268, 2000.

\bibitem[TW13]{W13}
R.~Thompson and W.~Walter.
\newblock {\em Ordinary Differential Equations}.
\newblock Graduate Texts in Mathematics. Springer New York, 2013.

\bibitem[Wal17]{walpuski:g2instantons}
Thomas Walpuski.
\newblock {$G_2$}-instantons, associative submanifolds and {F}ueter sections.
\newblock {\em Comm. Anal. Geom.}, 25(4):847--893, 2017.

\end{thebibliography}

\end{document}